\documentclass[a4paper,11pt,notitlepage]{amsart}
\usepackage{amsmath,amssymb,mathrsfs,amsthm}

\usepackage{verbatim}
\usepackage[spanish,USenglish]{babel} 
\usepackage{color}
\usepackage{graphicx}

\newtheorem{theorem}{Theorem}
\newtheorem{corollary}[theorem]{Corollary}
\newtheorem{lemma}[theorem]{Lemma}
\newtheorem{example}[theorem]{\it Example}
\newtheorem{proposition}[theorem]{Proposition}
\newtheorem{definition}[theorem]{Definition}

\newtheorem{remark}[theorem]{\it Remark}

\newcommand\Real{{\mathfrak R}{\mathfrak e}\,} %

\newcommand{\R}{\mathbb{{R}}}

\newcommand{\N}{\mathbb{{N}}}

\newcommand{\C}{\mathbb{{C}}}
\newcommand{\D}{\mathbb{{D}}}

\begin{document}

\title[Poisson equation and Hilbert transform]{Poisson equation and discrete one-sided Hilbert transform for $(C,\alpha)$-bounded operators}

\author[L. Abadias]{Luciano Abadias}
\address[L. Abadias]{Departamento de Matem\'aticas, Instituto Universitario de Matem\'aticas y Aplicaciones, Universidad de Zaragoza, 50009 Zaragoza, Spain.}
\email{labadias@unizar.es}

\author[J. E. Gal\'e]{Jos\'e E. Gal\'e}
\address[J. E. Gal\'e]{Departamento de Matem\'aticas, Instituto Universitario de Matem\'aticas y Aplicaciones, Universidad de Zaragoza, 50009 Zaragoza, Spain.}
\email{gale@unizar.es}

\author[C. Lizama]{Carlos Lizama}
\address[C. Lizama]{Departamento de Matem\'atica y Ciencia de la Computaci\'on, Facultad de Ciencias, Universidad de Santiago de Chile,
Casilla 307, Correo 2, Santiago, Chile}

\email{carlos.lizama@usach.cl}

\thanks{The two first-named authors have been partly supported by Project MTM2016-77710-P, DGI-FEDER, of the MCYTS and 
Project E26-17R, D.G. Arag\'on, Universidad de Zaragoza, Spain, and the third author has been supported by the CONICYT-Chile under 
FONDECYT grant number 1180041 and DICYT-Universidad de Santiago de Chile, USACH}



\begin{abstract}
We characterize the solutions of the Poisson equation and the domain of its associated one-sided Hilbert transform for 
Ces\`aro bounded operators of fractional order. The results obtained fairly generalize the corresponding ones for power-bounded operators. In passing, we give an extension of the mean ergodic theorem. Examples are given to illustrate the theory.
\end{abstract}

\date{}

\maketitle

\section{Introduction}\label{Intro}
\setcounter{theorem}{0}
\setcounter{equation}{0}

\medskip
Let $X$ be a complex Banach space and let $\mathcal{B}(X)$ denote the Banach algebra 
of bounded linear operators on $X$. 
Let $T\in\mathcal{B}(X)$. Then $T$  
is said to be power-bounded if 
$\sup\{\Vert T^n\Vert:n\in\N\}<\infty$. Put $M^1_T(n):=(n+1)^{-1}\sum_{j=0}^nT^jx$.  The operator $T$ is called 
Ces\`aro-mean bounded 
if $\sup\{\Vert M^1_T(n)\Vert:n\in\N\}<\infty$, and mean-ergodic if there exists
$P_1x:=\lim_{n\to\infty}M_T^1(n)x$ for all $x\in X$ (in norm). In this case $P_1$ is in fact a bounded projection onto 
the closed subspace 
${\rm{Ker}}(I-T)$ of $X$. Clearly, mean-ergodicity implies mean-boundedness.

Mean ergodic theorems form an important, classical, area of study since the beginning of the operator theory, 
see \cite[p. 657 and subseq.]{DS}. There are different versions of that type of theorems. 
The following result is 
well known. Let $T$ be a power-bounded operator on $X$. Then, for a given $x\in X$, there exists $P_{1}x$
if and only if $x$ belongs to the (topological) direct sum ${\rm{Ker}}(I-T)\oplus\overline{\rm{Ran}}(I-T)$ in $X$,
see \cite[Th. 1.3]{Kr}. Thus $T$ is mean-ergodic if and only if $X={\rm{Ker}}(I-T)\oplus\overline{\rm{Ran}}(I-T)$, which happens, for instance, when $X$ is reflexive \cite[Th. 1.2]{Kr}.

Ergodicity is related with probability theory via ergodic Markov chains
and their associated probability operators
$\mathcal P$ (involving pointwise convergence particularly,  when $X$ is taken as a $L_p$ space, for example); e.g.,
 it is possible to obtain central limit theorems for elements
$y\in\rm{Ran}(I-\mathcal P)\subseteq X$. Even more, with this kind of applications in mind and also looking at ergodicity in itself,
 it is relevant to find elements $x\in X$ for which the convergence rate in the mean ergodic limit
$$
P_1x=\lim_{n\to\infty}M_T^1(n)x
$$
is appropriate, say polynomial, for instance. It turns out that the above question is closely related with a suitable description of elements $y$ in range spaces
$(I-T)^sX\subseteq X$, $0<s\le1$, which takes us to the study of the
so-called fractional Poisson equation
\begin{equation}\label{PoissonEq}
(I-T)^sx=y, \quad x,y\in X.
\end{equation}
As a matter of fact, given $y\in X$, the equation (\ref{PoissonEq}) for a power-bounded and mean ergodic operator $T$ has a solution $x$
if and only if the series
$$
\sum_{n=0}^\infty{1\over n^{1-s}}T^ny
$$
converges in the norm of $X$, whence $\lim_{N\to\infty}\Vert n^{s-1}\sum_{n=1}^NT^ny\Vert=0$. Moreover, $x$ is obtained as a series 
representation $x=\sum_{n=1}^\infty c_n(s)T^ny$ where $c_n(s)\sim n^{s-1}$ as $n\to\infty$. 
For the above facts and other pertinent remarks, we refer to the introductions of references \cite{DL, HT} and \cite{GHT}.

In \cite{DL}, it is observed that $((I-T)^s)_{\Real s>0}$ is a (holomorphic) semigroup in ${\mathcal B}(X)$, and a $C_0$-semigroup if we further assume that $(I-T)X$ is dense in $X$. Let $\log(I-T)$ denote the infinitesimal generator of this semigroup. A natural question is whether or not $-\log(I-T)$ coincides with the one-sided ergodic Hilbert transform 
${\mathcal H}_T$ for $T$, given by
$$
{\mathcal H}_Tx:=\sum_{n=1}^\infty{T^n\over n}x,
$$
whenever $x\in X$ is such that the series converges in $X$.
Also in \cite{DL}, it was shown that ${\mathcal H}_T\subseteq\log(I-T)$ as (generally) unbounded operators on $X$. 
The equality 
${\mathcal H}_T=-\log(I-T)$ has been established in full generality in \cite{CCL} and \cite{HT} independently from one paper to the other and with different proofs  (see also \cite{AL, CL, CuL} for particular $T$). 
While the arguments used in \cite{CCL} are specific for series concerned with that one defining ${\mathcal H}_T$, 
the approach carried out in \cite{HT} relies on the usage of the functional calculus. Looking at power operators
 $(I-T)^s$, it is clear that the analysis of ${\rm{Ran}}(I-T)^s$ is equivalent to the study of the domain
${\rm{Dom}}(I-T)^{-s}$ of the inverse operator $(I-T)^{-s}$, $s>0$, where one is assuming that $I-T$ is injective. 
Also, the search for the equality
${\mathcal H}_T=-\log(I-T)$ entails the study of ${\rm{Dom}}\log(I-T)$. Then the idea leading \cite{HT} is to use the functional calculus for suitable analytic functions
$\frak f$, involving $\log(1-z)$ and $(1-z)^{-s}$, so that the elements of ${\rm{Dom}}\,\frak f(T)$ can be identified as those $x\in X$ which makes $\sum_{n=1}^\infty a_nT^nx$ norm-convergent, where $a_n$, $n\ge0$, are the Taylor coefficients 
of $\frak f$.

Initially, the basic domain of that calculus is the sequence space $\ell^1$, or its alternative version as (holomorphic) 
Wiener algebra $A_+^1(\D)$ on the unit disc $\D$. Then the domain is extended to so-called regularizable functions
$\frak f$, with respect to $A_+^1(\D)\equiv\ell^1$, in the way introduced in \cite{Haase}, see \cite[Def. 2.4]{HT}. 
But this is not enough to get characterized domains ${\rm{Dom}}\,{\frak f}(T)$. A sufficient condition on $\frak f$ is provided 
in \cite{HT} by introducing the notion of
{\it admissibility}, so that ${\rm{Dom}}\,{\frak f}(T)$ can be identified if $\frak f$ is both regularizable and admissible on 
$\D$.
By a theorem due to Th. Kaluza in 1928, remarkable examples of admissible functions are those whose Taylor coefficients form logarithmically convex sequences, see \cite[Prop. 4.4]{HT} (which gives a new proof of Kaluza's theorem).
Within this framework, the domains of  $(I-T)^{-s}$, $s>0$, and $\log(I-T)$ are characterized in \cite[Th. 6.1, Th. 6.2]{HT}.

In the present paper, we want to give a qualitative and quantitative jump  to that topic with a new  and original viewpoint, 
by noticing that the Poisson equation (\ref{PoissonEq}) can be also thought in the setting of fractional difference equations. 
Some authors have realized that the semigroup $(I-T)^s$ is a useful tool for modeling differential equations of fractional order, 
and it is well known that discretization techniques are useful in problems on differential equations; see \cite{AE, Li15} 
and references therein. As a sample, let $T$ be the backward shift operator on $\ell^2(\N_0)$. 
Then $I-T=D$,
where $D$ is the first order finite difference given by $Da(n):=a(n)-a(n+1)$, $n\ge0$, for every sequence $a$. 
Thus $D^2$ is the discretization of the one-dimensional Laplace operator. The discrete Poisson problem $D^{2s}u=f$ 
arises in the Markov chains theory, so that the corresponding operator $T$ is the transition matrix. The solution $u$ can be 
expressed as the asymptotic variance, which is an important parameter in central limit theorems (\cite{AsB, Gl, JLY, W}).
On the other hand, maximum and comparison principles for fractional differences $D^s$, as well as uniqueness of corresponding Dirichlet problems, have been recently established in connection with the elliptic problem $D^{s}u=f$ in
\cite{ADT}; a probabilistic interpretation of equation $D^su=0$ is also given in \cite[Remark 2.5]{ADT}.
Thus, answering the question of what are the elements of ranges $(I-T)^sX$ provides us with the kind of functions which are the proper data in problems on difference equations, so that it can be considered as an inverse problem.

In all of the above, the operators $T$ are assumed to be power-bounded. However, there are other important 
classes of operators with weaker growth properties on their powers for which the items discussed above are of interest.
This is the case of the $(C,\alpha)$-bounded operators defined below, in Section \ref{Pre}. 
The connection of these operators
and ergodicity dates back to the fourties of last century, see \cite{C} and \cite{Hi}. In the latter, E. Hille
studies $(C,\alpha)$-convergence in terms of Abel convergence (that is, via the resolvent operator). 
As an application, the well known mean ergodic von Neumann's theorem for unitary groups  on Hilbert spaces is extended to $(C,\alpha)$-convergence for every $\alpha>0$ \cite[p. 255]{Hi}. Also, the
$(C,\alpha)$-ergodicity on $L_1(0,1)$ of fractional (Riemann-Liouville) integrals is elucidated in \cite[Th. 11]{Hi}. 
In particular, if $V$ is the Volterra operator then $T_V:=I-V$, as operator on $L_1(0,1)$, is not power-bounded, 
and it is $(C,\alpha)$-ergodic if and only if $\alpha>1/2$ \cite[Th. 11]{Hi}.
As a matter of fact, growth properties and ergodicity of $(C,\alpha)$ means of operators have been extensively studied
over the years (see \cite{AB, De00, Ed04, E2, LSS, Sa, Su-Ze13, TZ, Yo98} and references therein). 
Very recently, in connection with operator inequalities and models, it has been shown that the shift operator on weighted Bergman spaces is $(C, \alpha)$-ergodic, for $\alpha>0$ depending on the weight. See \cite{ABY} and Section \ref{apliques} below.

In short, there is a well established literature on $(C,\alpha)$-bounded operators and ergodicity, which explores quite a number of properties and their interplays. However, 
neither the Poisson equation $(I-T)^sx=y$ nor the relation between the one-sided Hilbert transform and the logarithm operator 
$\log(I-T)$ seem to have been raised for
$(C,\alpha)$-bounded operators $T$ with fractional $\alpha>0$. 

Here, we generalize or extend results of \cite{DL, CCL, HT} to the setting of $(C,\alpha)$-bounded operators.
To do this, we follow the methodology introduced in \cite{HT} and, reflecting the viewpoint quoted formerly, recent tools associated with fractional differences \cite{ADT, Li17}.
The paper is organized accordingly.

After this introduction, Section \ref{Pre} is devoted to preliminaries on $(C,\alpha)$-bounded operators and fractional differences, which are defined in terms of Ces\`aro numbers. It is to be noticed that 
the Weyl difference operator $W^\alpha$ and its partner $D^\alpha$, $\alpha>0$, defined on sequence spaces, coincide sometimes but not always, this fact being one of the subtle difficulties to circumvent in the article. The section contains examples of further application.

The decomposition $X={\rm{Ker}}(I-T)\oplus\overline{{\rm{Ran}}}(I-T)$ plays a key role in ergodic theorems, and is also relevant in the treatment of the (fractional) Poisson equation and the one-sided (discrete) Hilbert transform. Thus this item is discussed in
Section \ref{ergodicity} for $(C,\alpha)$-bounded $T$. In Theorem \ref{MeanErg}, it is shown that for such an operator $T$ the above splitting of $X$ occurs if and only if $T$ is $(C,\beta)$-ergodic for every $\beta>\alpha$. This result is a fairly general extension of the mean ergodic theorem for power-bounded operators.

In order to set our discussion within the framework built in \cite{HT}, we need to replace the $\ell^1$-calculus 
with another one which characterizes
$(C,\alpha)$-bounded operators. The domain of this calculus is the convolution subalgebra $\tau^\alpha$ of $\ell^1$ formed by sequences $f\in\ell^1$ such that the series
$\Vert f\Vert_{(\alpha)}:=\sum_{n=0}^\infty\vert W^\alpha f(n)\vert k^{\alpha+1}(n)$ is finite (see \cite{ALMV}). The holomorphic function counterpart of
$\tau^\alpha$, which is to say the space of Taylor series of elements in $\tau^\alpha$, is denoted by $A^\alpha(\D)$ 
(note that the algebra $A^1_+(\D)$ of \cite{HT} is $A^0(\D)$ here). Section \ref{CFcesar} contains basic properties of the 
$\tau^\alpha$-calculus (or $A^\alpha(\D)$-calculus, equivalently), its relation with fractional difference operators 
and the coincidence of this calculus with that one for sectorial operators via the transformation $z\mapsto(1-z)$. Moreover, 
$\alpha$-regularizable functions are defined according to \cite{Haase}, with respect to 
the algebra $A^\alpha(\D)$ (Definition \ref{a-reg}). In the examples, it is shown that functions 
$(1-z)^s$ and $\log(1-z)$ are
$\alpha$-regularizable. Finally, the identity
$\overline{{\rm{Ran}}}(I-T)^s=\overline{{\rm{Ran}}}(I-T)$ is proved for $(C,\alpha)$-bounded operators $T$.
A generalization of admissibility, called here $\alpha$-admissibility (Definition \ref{a-admisi}), is implemented in
Section \ref{admisi}. Its relation with fractional differences is analysed, as well as the possibility 
to construct certain approximate units out from $\alpha$-admissible functions. Having laid the technical groundwork 
of our paper in Section \ref{admisi}, examples of $\alpha$-admissibility are presented in
Section \ref{exadmisi}. These are generalizations of logarithmic convexity to higher order differences and fractional differences. 
Then an extension of Kaluza's result is provided (Theorem \ref{logconvex}) that allows us to include as concrete examples 
the functions $(1-z)^{-s}$ and $-z^{-1}\log(1-z)$, and so operators $(I-T)^{-s}$ and $\log(1-T)$ in the 
$(C,\alpha)$-operator context.

With the above stuff at disposition, we give in Section \ref{DOF} a characterization of $x\in{\rm{Dom}}\,{\frak f}(T)$ 
by convergence of the series
$\sum_{n=0}^\infty D^\alpha f(n)\Delta^{-\alpha}{\mathcal T}(n)x$, where 
$\Delta^{-\alpha}{\mathcal T}(n)=\sum_{j=0}^n k^{\alpha}(n-j)T^j$
(Theorem \ref{dominios}). 
 In particular, in Section \ref{PoiLo} it is shown that for $0<s<1$ one has
$x\in{\rm{Dom}}(I-T)^{-s}$ if and only if $\sum_{n=1}^\infty n^{s-\alpha-1}\Delta^{-\alpha}{\mathcal T}(n)x$ converges 
(Theorem \ref{CesPoisson}), and $x\in{\rm{Dom}}\log(I-T)$ if and only if the series
$\sum_{n=1}^\infty n^{-(1+\alpha)}\Delta^{-\alpha}{\mathcal T}(n)x$ converges and, in this case, we obtain the new formula
$$
\log(I-T)x=(\psi(\alpha+1)-\psi(1)) x
-\sum_{n=1}^\infty B(\alpha + 1,n) \Delta^{-\alpha}\mathcal{T}(n)x,
$$
where $\psi$ is the digamma function -so that $\psi(\alpha+1)-\psi(1)=\int_0^1(1-u^\alpha)(1-u)^{-1}du$- and B is the Beta function (Theorem \ref{LOG}). This latter result suggests that one might perhaps define the one-sided 
{\it$\alpha$-ergodic} Hilbert transform for a $(C,\alpha)$-bounded operator $T$ by
$$
H_T^{(\alpha)}:=(\psi(1)-\psi(\alpha + 1))
+\sum_{n=1}^\infty B(\alpha + 1, n) \Delta^{-\alpha}\mathcal{T}(n)
$$
(for $\alpha=0$ it equals the usual one-sided Hilbert transform $H_T$).

However, the actions of $(I-T)^{-s}$ and $H_T^{(\alpha)}$, on elements in their corresponding domains, 
keep their original forms in terms of Taylor series 
when $0<\alpha<1-s$ and 
$0<\alpha<1$ respectively. 
To show this, the arguments of Section \ref{admisi}, worked out on coefficients 
$W^\alpha f(n)$ and $D^\alpha f(n)$, are refined in Section \ref{taylor} to find suitable approximate units for 
$A^\alpha(\D)$ in terms of Taylor coefficients (see Theorem \ref{approximation01} and Theorem \ref{approximation02}). 

In this way, we obtain, for $0<s<1$ and $0<\alpha<1-s$ (Theorem \ref{equivTaylor}), the interesting characterization
$$
x\in{\rm{Dom}}(I-T)^{-s}\Longleftrightarrow\sum_{n=1}^\infty n^{s-1}T^nx\, \hbox{ converges } 
$$
and, for $0<\alpha<1$, Theorem \ref{LOG&hilbert}:
$$
x\in{\rm{Dom}}\log(I-T)\Longleftrightarrow
\sum_{n=1}^\infty n^{-1}T^nx\, \hbox{ converges}.
$$
Moreover in this case $\log(I-T)=-{\mathcal H}_T$, that is,
$$
\log(I-T)=-\sum_{n=1}^\infty{1\over n}T^n.
$$
Theorem \ref{equivTaylor} extends \cite[Th. 2.11]{DL} and \cite[Th. 6.1]{HT}, and Theorem \ref{LOG&hilbert} extends 
\cite[Prop. 3.3]{CCL} and \cite[Th. 6.2]{HT}.

In section \ref{apliques} we show two concrete examples to illustrate and apply the preceding results. These examples are about the Volterra operator on $L_p$ spaces, and about the shift operator on Bergman spaces expressed in the form of sequences. 
Finally, the proof of (extended Kaluza's) Theorem \ref{logconvex} has been relegated to an Appendix.

\section{Preliminaries}\label{Pre}

For $\alpha\in\C$, let $k^{\alpha}$ denote the sequence on
$\N_0:=\{0,1, 2, \dots \}$ where  $k^{\alpha}(n)$ is the $n$-th coefficient of the generating function
$(1-z)^{-\alpha}$; that is,
\begin{equation}\label{generating}
\displaystyle\sum_{n=0}^{\infty}k^{\alpha}(n)z^n=\frac{1}{(1-z)^{\alpha}},\quad |z|<1.
 \end{equation}
The elements of sequences $k^{\alpha}$ are called Ces\`aro numbers, and are given by $k^{\alpha}(0)=1$ and
\begin{equation}\label{definition}
k^{\alpha}(n)
:={n+\alpha-1\choose \alpha-1}=\frac{\alpha(\alpha+1)\cdots(\alpha+n-1)}{n!},\, n\in\N;
\end{equation}
see \cite[Vol. I, p.77]{Zygmund}, where $k^{\alpha}(n)$ is denoted by $A_{n}^{\alpha-1}$. 
For
$\alpha\in \C\setminus\{0,-1,-2,\dots\}$ one has
$\displaystyle k^{\alpha}(n)=\frac{\Gamma(n+\alpha)}{\Gamma(\alpha)\Gamma(n+1)},$
where $\Gamma$ is the Gamma function. 
It is well known the important role played by sequences $k^{\alpha}$ in the theory of summability of Fourier series, see \cite{Zygmund}. Recently, it has been realized that they have application in the theory of fractional difference equations; 
i. e. \cite{GL, Li15, Li17}.

It follows from \eqref{generating} that the family $(k^{\alpha})_{\alpha\in\C}$
satisfies the group property, $k^{\alpha}\ast k^{\beta}=k^{\alpha+\beta}$ for $\alpha, \beta \in \C$, where the operation
$\lq\lq\ast"$ is the convolution of sequences. Note that $k^0$ is the Dirac mass $\delta_0$ on $\N_0$. 
Also by \eqref{generating},
\begin{equation}\label{zero}
\sum_{n=0}^{\infty}k^{\alpha}(n)=0,\quad \hbox{for }\, \Real \alpha<0.
\end{equation}

Assume $\alpha\in\R$. As a function of $n$, $k^\alpha$ is increasing  for $\alpha >1$, decreasing for 
$0<\alpha <1,$ and $k^1(n)=1$ for all 
$n\in \N_0$
(\cite[Th. III.1.17]{Zygmund}). Furthermore, $0\leq k^\alpha(n)\le k^\beta(n)$ for $\beta \ge \alpha>0$ and $n\in \N_0$. For $m\in\N_0,$

\begin{equation}\label{binom}
k^{-m}(n)=\left\{\begin{array}{ll}
(-1)^n\binom{m}{n},& 0\leq n\leq m; \\
0,& n\geq m+1,
\end{array} \right.
\end{equation}

and, if $m< \alpha<m+1,$

\begin{equation}\label{sign}
\text{sign}\,k^{-\alpha}(n)=\left\{\begin{array}{ll}
(-1)^n,& 0\leq n\leq m; \\
(-1)^{m+1},& n\geq m+1.
\end{array} \right.
\end{equation}

As regards the asymptotic behaviour of the sequence $k^{\alpha}$ we have
\begin{equation}\label{double2}
k^{\alpha}(n)=\frac{n^{\alpha-1}}{\Gamma(\alpha)}(1+O({1\over n})),\ \text{ as }n\to\infty, \, \hbox{ for every }
\alpha\in \C\backslash\{0,-1,-2,\ldots\};
\end{equation}
see \cite[Vol.I, p.77 (1.18)]{Zygmund} or \cite[Eq.(1)]{ET}.

\medskip
Ces\`aro numbers appear in the definition of {\it fractional differences}. For a sequence $f$, define
$$
Wf(n):=f(n)-f(n+1),\quad n\in \N_0,
$$

and subsequently, $W^1:=W$,  $W^{m+1}:= W^mW$ for $m\in\N$. Then one has
$$
W^{m}f(n)=\displaystyle\sum_{j=0}^{m}(-1)^{j}\binom{m}{j}f(n+j),
\quad n\in \N_0,\, m\in \N.
$$

Differences $W^m$ are extended to the fractional case in \cite[Def. 2.2]{ALMV} as follows.

\begin{definition}\label{WeylDifference}
\normalfont
Let $f: \mathbb{N}_0 \to \C$  and $\alpha>0$ be given.
The  {\it Weyl sum} $W^{-\alpha}f$ {\it of order} $\alpha$ of $f$ is defined by
$$
W^{-\alpha}f(n):=\displaystyle\sum_{j=n}^{\infty} k^{\alpha}(j-n)f(j),\quad n\in\N_0,
$$
whenever the right hand side makes sense.
The {\it Weyl difference} $W^\alpha f$ {\it of order} $\alpha$ of $f$ is defined by
$$
W^{\alpha}f(n)=W^m W^{-(m-\alpha)}f(n),\quad n\in\N_0,
$$
for $m=[\alpha]+1$, with $[\alpha]$ the integer part of $\alpha$, whenever the right hand side makes sense.
\end{definition}

In some cases, $W^\alpha$ admits another useful description: Let $\ell^1(\omega)$ denote the space of absolutely summable sequences on $\N_0$ with respect to a weight $\omega\colon\N_0\to\C$. Let $D^{\alpha}$ be the operator on sequences $f$ given formally by
$$
D^{\alpha}f(n):=\displaystyle\sum_{j=n}^{\infty} k^{-\alpha}(j-n)f(j), \quad n\in\N_0;
$$
see \cite{ADT}.
Note that, for $\alpha=m\in\N_0$,
$$
D^m f(n)=\displaystyle\sum_{j=n}^{\infty} k^{-m}(j-n)f(j)=
\sum_{j=0}^{m} (-1)^{j}\binom{m}{j}f(j+n)=W^mf(n)
$$
for all $f$ and $n$. Moreover, by
\cite[Th. 3.2]{AM}\label{Theorem 2.2}
we have
\begin{equation}\label{WigualD}
W^{\alpha}f=D^{\alpha}f
\end{equation}
whenever both expressions make sense.
For example, it happens when
$f\in\ell^1(k^{m-\alpha})\subset \ell^1(k^{-\alpha})$, $m=[\alpha]+1$. To see this, notice that $W^\alpha$ is well defined on
$\ell^1(k^{m-\alpha})$, $D^\alpha$ is well defined on $\ell^1(k^{-\alpha})$ and
$\ell^1(k^{m-\alpha})\hookrightarrow\ell^1(k^{-\alpha})$. (In the opposite, the two members of \eqref{WigualD} need not make sense simultaneously. For instance, if $0<\alpha<1$ the series
$\sum_{j=n}^{\infty} k^{-\alpha}(j-n)k^1(j)$ is convergent, but $W^{\alpha}k^1$ is not defined.)
Also, it is clear that
$W^{-\alpha}$ is well defined on $\ell^1(k^\alpha)$. Moreover it is injective on $\ell^1(k^\alpha)$, since
 \begin{equation}\label{deuvedoble}
 D^\alpha(W^{-\alpha}f)=f,\ \text{ for every } f\in\ell^1(k^\alpha).
 \end{equation}

In effect, first note that $(\vert k^{-\alpha}\vert\ast k^\alpha)(q)=O(k^\alpha(q))$, as $q\to\infty$, $\alpha>0$. To see this,  take the integer part $[\alpha]$ of $\alpha$ and $q>[\alpha]$.
Then by (\ref{sign}), $(\vert k^{-\alpha}\vert\ast k^\alpha)(q)=\sum_{p=0}^q\vert k^{-\alpha}(p)\vert k^{\alpha}(q-p)
=2k^{\alpha}(q)-\sum_{p=0}^q k^{-\alpha}(p)k^{\alpha}(q-p)=2k^{\alpha}(q)-k^0(q)=2k^{\alpha}(q)$
if  $0<\alpha<1$, and
\begin{eqnarray*}
(\vert k^{-\alpha}\vert&\ast& k^\alpha)(q)=\sum_{p=0}^{[\alpha]}(-1)^p k^{-\alpha}(p)k^{\alpha}(q-p)
+\sum_{p=[\alpha]+1}^q (-1)^{[\alpha]+1}k^{-\alpha}(p)k^{\alpha}(q-p)\\
&=&
\sum_{p=0}^{[\alpha]}\left((-1)^p+(-1)^{[\alpha]}\right) k^{-\alpha}(p)k^{\alpha}(q-p)
+(-1)^{[\alpha]+1}  \sum_{p=0}^{q} k^{-\alpha}(p)k^{\alpha}(q-p)\\
&\le& 2\sum_{p=0}^{[\alpha]}\vert k^{-\alpha}(p)\vert k^{\alpha}(q-p) + (-1)^{[\alpha]+1} k^0(q)\le M_\alpha k^\alpha(q),
\end{eqnarray*}
for some positive constant $M_\alpha$, if $\alpha\ge1$. Hence, for $n\in\N_0$ and $f\in\ell^1(k^\alpha)$,
$$
\vert D^\alpha(W^{-\alpha}f)(n)\vert\le 
\sum_{j=0}^{\infty}|k^{-\alpha}(j)|\sum_{l=n+j}^{\infty}k^\alpha(l-n-j)|f(l)|
\leq M_{\alpha}\sum_{l=n}^{\infty}|f(l)|k^{\alpha}(l-n)<\infty,
$$
which implies in turn, for $n\in\N_0$,
$$
D^\alpha(W^{-\alpha}f)(n)=\sum_{l=n}^{\infty}f(l)\sum_{j=0}^{l-n}k^{\alpha}(l-n-j)k^{-\alpha}(j)
=\sum_{l=n}^{\infty}f(l)\delta_{0}(l-n)=f(n),
$$
as we wanted to show.

\begin{example}\label{ex22}
{\rm\begin{itemize}
\item[(i)] For $\mu \in \mathbb{C}\backslash\{0\}$ define $p_{\mu}(n):=\mu^{-(n+1)}$, $n\in\N_0$. 
It is proven in \cite[Ex. 2.5]{ALMV} that functions $p_{\mu}$ are eigenfunctions of the operator 
$W^{\alpha}$ for $\alpha\in \R$ and $\vert \mu\vert >1;$ namely,
$$
W^{\alpha}p_{\mu}= \mu^{-\alpha}(\mu-1)^{\alpha}p_{\mu},\quad |\mu|>1.
$$
\item[(ii)] Let $s\in \R$ and $m\in\N_0.$ Then
$$
D^m k^{s}(n)=W^m k^{s}(n)
=(-1)^{m}k^{s-m}(n+m), \quad n\in\N_0;
$$
see \cite[Ex. 3.4]{AM}. Also, if $\alpha>0$ and $s\in (0,1),$ then by \cite[Lemma 1.1]{AM} one gets 
\begin{equation}\label{2.10}D^{\alpha} k^s(n)=\frac{\Gamma(1-s+\alpha)\Gamma(s+n)}{\Gamma(s)\Gamma(1-s)\Gamma(n+\alpha+1)}=\frac{\sin(\pi s)\Gamma(1-s+\alpha)\Gamma(s+n)}{\pi\Gamma(n+\alpha+1)},\quad n\in\N_0,
\end{equation}
and therefore, by \cite[Eq.(1)]{ET}, 
\begin{equation}
\label{2.11}D^{\alpha} k^s(n)
=\frac{\Gamma(1-s+\alpha)}{\Gamma(s)\Gamma(1-s)}n^{s-\alpha-1}(1+O({1\over n})), \quad\text{ as }n\to\infty.
\end{equation}

\item[(iii)] Let $m\in\N_0$ and let $L$ be the sequence defined by $L(n)=\displaystyle\frac{1}{n+1}$ for $n\in\N_0$.
Then,
for $\alpha>0,$ by \cite[Lemma 1.1]{AM} we have
\begin{equation}\label{2.12}
D^{\alpha}L(n)=\frac{\Gamma(\alpha+1)n!}{\Gamma(n+\alpha+2)},\quad n\in\N_0,
\end{equation}
and by \cite[Eq.(1)]{ET},
\begin{equation}\label{2.13}
D^{\alpha}L(n)=\frac{\Gamma(\alpha+1)}{n^{\alpha+1}}(1+O({1\over n})), \quad\text{ as }n\to\infty.
\end{equation}
\end{itemize}}
\end{example}

\medskip
We now introduce $(C,\alpha)$-bounded operators and ergodicity in terms of 
Ces\`aro numbers.  
Let $X$ be a Banach space. 
For $T$ in $\mathcal{B}(X)$, let
$\mathcal{T}$ denote the discrete semigroup associated with $T$, given by
$\mathcal{T}(n):=T^n$ for $n\in\N_0:=\N\cup\{0\}$, with $T^0$ the identity operator on $X$.
Take $\alpha\ge0$ and set, for $x\in X$ and $\ n\in \N_0$,
\begin{equation*}
\Delta^{-\alpha} \mathcal{T}(n)x :=(k^{\alpha}*\mathcal{T})(n)
= \displaystyle\sum_{j=0}^n k^{\alpha}(n-j)T^j x,\qquad M_T^{\alpha}(n)x
:=\frac{1}{k^{\alpha+1}(n)}\Delta^{-\alpha} \mathcal{T}(n)x.
\end{equation*}
The operators $\Delta^{-\alpha} \mathcal{T}(n)$ and
$M^{\alpha}_T(n)$ in $\mathcal{B}(X)$ are called the $n$-th
Ces\`{a}ro sum and Ces\`{a}ro mean of $T$ of order $\alpha$, respectively.

The operator $T$ is called Ces\`{a}ro bounded of order $\alpha$, or simply $(C,\alpha)$-bounded, if it satisfies
$\sup_{n} \|M_T^{\alpha}(n) \| < \infty$. Thus
$(C,0)$-boundedness is the same as
power-boundedness, that is, $\sup_{n}\Vert T^n\Vert<\infty$.
For $\alpha=1$ the operator $T$ is called Ces\`{a}ro mean bounded (or Ces\`{a}ro bounded). If $T$ is $(C,\alpha)$-bounded then it is $(C,\beta)$-bounded for every $\beta>\alpha$ but the converse does not hold true in general;
 see for example \cite[Section 4.7]{E2} , \cite[Remark 2.3]{Su-Ze13} and \cite[Section 2]{AB}.

On the other hand, a bounded linear operator $T$ is said to be $(C,\alpha)$-ergodic
if there exists
$P_{\alpha}x:=\lim_{n\to\infty}M_T^{\alpha}(n)x$ for all $x\in X$,
in the norm of $X$ (in this case $P_{\alpha}$ is in fact a bounded projection onto the closed subspace ${\rm{Ker}}(I-T)$ of $X$). For $\alpha=1$, $T$ is called mean ergodic.
Clearly, $(C,\alpha)$-ergodicity implies $(C,\alpha)$-boundedness. Quite a number of papers have been done about ergodicity of $(C,\alpha)$-bounded operators,
and about the growth of Ces\'aro sums and Ces\'aro means of order $\alpha$,
which look for extending to fractional order the main results and features of the operator ergodic theory. We focus here on the particular line of research explained in the former introduction, Section \ref{Intro}, which seems to have not been considered before. Nonetheless, we establish a number of ergodic results on $(C,\alpha)$-bounded operators in the next section.

\section{$(C,\alpha)$-mean ergodic results}\label{ergodicity}
\setcounter{theorem}{0}
\setcounter{equation}{0}

A version of the mean ergodic theorem says that a power-bounded operator $T$ is Ces\`{a}ro mean ergodic if and only if $X$ splits as
$X={\rm {Ker}}(I-T)\oplus\overline{{\rm{Ran}}}(I-T)$ \cite[Th. 1.3]{Kr}.
Our aim here is to give an extension of that result for $(C,\alpha)$-bounded operators for every $\alpha>0$. Recall that $T$ is $(C,\alpha)$-bounded when $\sup_{n} \|M_T^{\alpha}(n) \| < \infty$, where
$k^{\alpha+1}(n)M_T^{\alpha}(n):=\Delta^{-\alpha} \mathcal{T}(n):=(k^{\alpha}*\mathcal{T})(n)
=\sum_{j=0}^n k^{\alpha}(n-j)T^j$.

\begin{lemma}\label{dondeconverge}
Let $T$ be a $(C,\alpha)$-bounded operator on  a Banach space $X,$ and $\beta>\alpha$. Put
$X_{\beta}:=\{x\in X: \hbox{ \rm{there exits} } P_\beta x:=\lim_{n\to\infty}M_T^{\beta}(n)x \hbox{ \rm{in} } X\}$.
Then
$$
{\rm{Ran}} P_\beta ={\rm {Ker}}(I-T)\, \hbox{ and }\, {\rm {Ker}} P_\beta =\overline{{\rm{Ran}}}(I-T),
$$
so that
$$
X_{\beta}={\rm {Ker}}(I-T)\oplus\overline{{\rm{Ran}}}(I-T).
$$
\end{lemma}

\begin{proof}
First, note that 
$T\Delta^{-\beta}\mathcal{T}(n)=\Delta^{-\beta}\mathcal{T}(n+1)-k^{\beta}(n+1)I,$ 
which implies  
$$TM_T^{\beta}(n)=\frac{\beta+n+1}{n+1}M_T^{\beta}(n+1)-\frac{\beta}{n+1}I.$$

For a given $x\in X_{\beta}$, letting $n\to\infty$ one obtains that there exists $P_\beta Tx=TP_\beta x=P_{\beta}x$. 
Hence, $P_\beta M_T^{\beta}(n)x=M_T^{\beta}(n)  P_\beta x
  = \displaystyle{1\over k^{\beta+1}(n)}\sum_{j=0}^n k^{\beta}(n-j)T^jP_\beta x
 ={1\over k^{\beta+1}(n)}(k^{\beta}\ast k^1)(P_\beta x)=P_\beta x$,
which implies that $P_\beta^2x=P_\beta x$. Therefore $P_\beta$ is a (linear) bounded projection on $X_{\beta}$.

Now, for $x\in X_{\beta}$,
$(I-T)P_\beta x=0$ so ${\rm{Ran}} P_\beta\subseteq{\rm {Ker}}(I-T)$. Conversely, if $Tx=x$ then
$M_T^{\beta}(n)x=x$ for all $n$, and therefore there exists $ P_\beta x=x$ in $X_{\beta}$. In short,
${\rm{Ran}} P_\beta ={\rm {Ker}}(I-T)$.

To see that ${\rm {Ker}} P_\beta =\overline{{\rm{Ran}}}(I-T)$, note that by Example \ref{ex22} (ii) one gets  
\begin{eqnarray*}\label{Media}
(I-T)\Delta^{-\beta}\mathcal{T}(n)&=&(I-T)(k^{\beta-\alpha}*k^{\alpha}*\Delta^{-\alpha}\mathcal{T})(n) \nonumber\\
&=&\sum_{j=0}^nk^{\beta-\alpha}(n-j)\Delta^{-\alpha}\mathcal{T}(j)-\sum_{j=0}^nk^{\beta-\alpha}(n-j)(\Delta^{-\alpha}\mathcal{T}(j+1)-k^{\alpha}(j+1)I) \nonumber\\
&=&k^{\beta-\alpha}(n)+\sum_{j=0}^nk^{\beta-\alpha}(n-j)k^{\alpha}(j+1)-\sum_{j=1}^{n+1}k^{\beta-\alpha-1}(n+1-j)\Delta^{-\alpha}\mathcal{T}(j)\\
&=&k^{\beta-\alpha}(n)+k^{\beta}(n+1)-k^{\beta-\alpha}(n+1)+k^{\beta-\alpha-1}(n+1) \nonumber \\
&&-\sum_{j=0}^{n+1}k^{\beta-\alpha-1}(n+1-j)\Delta^{-\alpha}\mathcal{T}(j) \nonumber \\
&=&k^{\beta}(n+1)-\sum_{j=0}^{n+1}k^{\beta-\alpha-1}(n+1-j)\Delta^{-\alpha}\mathcal{T}(j).\nonumber
\end{eqnarray*}
On the one hand $\displaystyle\frac{k^{\beta}(n+1)}{k^{\beta+1}(n)}\to 0$ as $n\to\infty.$ 
On the other hand, if $\beta\geq \alpha+1,$ by the $(C,\alpha)$-boundedness of $T,$ 
$\displaystyle\frac{1}{k^{\beta+1}(n)}\lVert \sum_{j=0}^{n+1}k^{\beta-\alpha-1}(n+1-j)\Delta^{-\alpha}\mathcal{T}(j)\rVert 
\leq \frac{k^{\beta}(n+1)}{k^{\beta+1}(n)}\to 0,\quad n\to\infty.$

If $\beta-\alpha\in (0,1),$ by \eqref{sign} we have 
\begin{displaymath}
\begin{array}{l}
\displaystyle\frac{1}{k^{\beta+1}(n)}\lVert \sum_{j=0}^{n+1}k^{\beta-\alpha-1}(n+1-j)\Delta^{-\alpha}\mathcal{T}(j)\rVert 
\leq\frac{1}{k^{\beta+1}(n)} \sum_{j=0}^{n+1}|k^{\beta-\alpha-1}(n+1-j)|k^{\alpha+1}(j)\\
\displaystyle=\frac{1}{k^{\beta+1}(n)}
\biggl(-\sum_{j=0}^{n+1}k^{\beta-\alpha-1}(n+1-j)k^{\alpha+1}(j)+2k^{\alpha+1}(n+1)  \biggr)\\
\displaystyle=\frac{1}{k^{\beta+1}(n)}(-k^{\beta}(n+1)+2k^{\alpha+1}(n+1))\to 0,\quad n\to\infty.
\end{array}
\end{displaymath}
So, we conclude that $M_T^{\beta}(n)(I-T)\to0$ as $n\to\infty$ for any $\beta>\alpha$. 
It follows that ${\rm{Ran}}(I-T)\subseteq X_{\beta}$, with $P_\beta(I-T)=0$ indeed. That is,
${\rm{Ran}}(I-T)\subseteq{\rm {Ker}} P_\beta$ and then $\overline{{\rm{Ran}}}(I-T)\subseteq{\rm {Ker}} P_\beta$.
Conversely, given $x$ in $X_{\beta}$ such that $P_\beta x=0$ then
 $x=\displaystyle\lim_{n\to\infty} {1\over k^{\beta+1}(n)}\sum_{j=0}^n k^{\beta}(n-j)(x-T^j x)
=\displaystyle\lim_{n\to\infty}(I-T)\left(\sum_{j=1}^n
 \displaystyle{k^{\beta}(n-j)\over k^{\beta+1}(n)}\sum_{l=0}^{j-1}T^l x\right)$,
 that is, ${\rm {Ker}} P_\beta \subseteq\overline{{\rm{Ran}}}(I-T)$ and the proof is over.
 \end{proof}

\begin{remark}\label{propSplit}
\normalfont
In accordance with Lemma \ref{dondeconverge}, for a given $(C,\alpha)$-bounded operator $T$ the direct sum
${\rm {Ker}}(I-T)\oplus\overline{{\rm{Ran}}}(I-T)$ is the largest subspace of $X$ on which there exists
$\lim_{n\to\infty}M_T^{\beta}(n)$, for each $\beta>\alpha.$
\end{remark}

The following result is the $(C,\alpha)$ mean ergodic extension quoted above.

\begin{theorem}\label{MeanErg}
Let $T$ be a $(C,\alpha)$-bounded operator, and $\beta>\alpha$. Then $T$ is $(C,\beta)$-ergodic if and only if
$X={\rm {Ker}}(I-T)\oplus\overline{{\rm{Ran}}}(I-T)$.
\end{theorem}
\begin{proof}
It is enough to have into account Remark \ref{propSplit}.
\end{proof}

The following immediate corollary shows the significance of the above theorem.

\begin{corollary}\label{CorMeanErg} Let $T$ be a bounded operator on a Banach space $X$.
\begin{itemize}
\item[(i)] If $T$ is $(C,\alpha)$-ergodic on $X$ for some $\alpha>0$ then
$X={\rm {Ker}}(I-T)\oplus\overline{{\rm{Ran}}}(I-T)$.
\item[(ii)] If $T$ is power-bounded then $T$ is $(C,\beta)$-ergodic if and only if
$X={\rm {Ker}}(I-T)\oplus\overline{{\rm{Ran}}}(I-T)$ for every $\beta>0$.
\item[(iii)] Let $T$ be a $(C,\alpha)$-bounded operator with $0<\alpha<1$. Then $T$ is Ces\`aro mean ergodic if and only if $X={\rm {Ker}}(I-T)\oplus\overline{{\rm{Ran}}}(I-T)$.
\end{itemize}
\end{corollary}

When $\beta=1$, Corollary \ref{CorMeanErg} (ii) is the well known mean ergodic theorem cited in the beginning of this section.

\medskip
Next, we extend \cite [Theorem 5.1]{Su-Ze13} to the fractional case $\beta\geq 1$. The proof runs parallel to that one of \cite[Th. 5.1]{Su-Ze13}, though it needs \cite[Th. 4.3]{A} 
in our case.

\begin{theorem}\label{Character}
Let $\beta\geq 1$ and $T\in\mathcal{B}(X)$ such that $\sigma(T)\subset \mathbb{D}\cup \{1\}$. Then $T$ is a $(C,\beta)$-ergodic operator
if and only if $T$ is $(C,\beta)$-bounded and
$X={\rm {Ker}}(I-T)\oplus\overline{{\rm{Ran}}}(I-T)$.
\end{theorem}
\begin{proof}
The implication which assumes the $(C,\beta)$-ergodicity of $T$ as hypothesis is clear. Now, let $T$ be a $(C,\beta)$-bounded operator on $X={\rm {Ker}}(I-T)\oplus\overline{{\rm{Ran}}}(I-T)$ with $\sigma(T)\subset \mathbb{D}\cup \{1\}$. Every $x\in X$ has the expression
$x=y+z$ with $y\in\overline{{\rm{Ran}}}(I-T)$ and $Tz=z$. Then it is enough to show that $M_T^{\beta}(n)y\to 0$ as $n\to\infty$ to prove the theorem.

Take $y=a-Ta$, $a\in X$. Then
$$
M_T^{\beta}(n)(a-Ta)=\displaystyle{\beta\over n+1}(I-M_T^{\beta-1}(n+1))a\to0,\quad \hbox{ as } n\to\infty
$$
since $\sigma(T)\subset \mathbb{D}\cup \{1\}$ and therefore $\Vert M_T^{\beta-1}(n+1)\Vert=o(n)$, as $n\to\infty$; see \cite[Th. 4.3]{A}. By density, one obtains
$\lim_{n\to\infty}M_T^{\beta}(n)y=0$ for all $y\in\overline{{\rm{Ran}}}(I-T)$.
\end{proof}

\begin{remark}
\normalfont
It can be shown directly from definitions that if $T$ is a $(C,\alpha)$-ergodic operator then $T$ is $(C,\beta)$-ergodic for every
$\beta>\alpha$. (This result is usually proved as a consequence of mean ergodic results involving the resolvent function of the operator; see \cite[Cor. 3.1]{Ed04} for example.) In this case, the projection operators $P_{\alpha}$ and $P_\beta$ are the same.
\end{remark}

\section{Functional calculus for Ces\`aro bounded operators}\label{CFcesar}
\setcounter{theorem}{0}
\setcounter{equation}{0}

Ces\`aro sums of order $\alpha$ enjoy an interesting multiplicative structure, see \cite[Th. 3.3]{ALMV}. In fact, the $(C,\alpha)$-boundedness is characterized by the existence of a bounded algebra homomorphism of a certain  weighted convolution 
Banach algebra of sequences into algebras of bounded linear operators (\cite[Cor. 3.7]{ALMV}). A copy of such a 
Banach algebra will serve as domain of a functional calculus, suitable  for our aims here.

Let us consider the action of $W^{-\alpha}$ on $\ell^1(k^{\alpha+1})$. We have
$$
\sum_{n=0}^\infty \vert W^{-\alpha}f(n)\vert
\le\sum_{n=0}^\infty\sum_{j=n}^\infty k^\alpha(j-n)\vert f(j)\vert
=\sum_{j=0}^\infty\sum_{n=0}^jk^{\alpha}(j-n)\vert f(j)\vert
=\sum_{j=0}^\infty k^{\alpha+1}(j)\vert f(j)\vert,
$$
for $f\in\ell^1(k^{\alpha+1})$ and $n\in\N_0$. In other words, $W^{-\alpha}\colon\ell^1(k^{\alpha+1})\to\ell^1$
is well defined and continuous.
Also, it is injective by (\ref{deuvedoble}) since $\ell^1(k^{\alpha+1})\hookrightarrow\ell^1(k^{\alpha})$.
Let $\tau^{\alpha}$ denote the Banach space
$W^{-\alpha}(\ell^1(k^{\alpha+1}))$ endowed with the norm transferred from that one of
$\ell^1(k^{\alpha+1})$. That is,
$\tau^{\alpha}$ is formed by the complex sequences $f$ of $\ell^1$ for which there exists a unique sequence
$W^\alpha f$ in $\ell^1(k^{\alpha+1})$ such that the series
$$
\|f\|_{(\alpha)}:=\sum_{n=0}^\infty |W^\alpha f(n)|k^{\alpha+1}(n)
$$
converges, and $\|f\|_{(\alpha)}$ is its norm. In this way,
$W^\alpha\colon\tau^\alpha\to\ell^1(k^{\alpha+1})$ is a surjective isometry with inverse
$W^{-\alpha}$. Note that $W^\alpha=D^\alpha$ on $\tau^\alpha$. Since $W^{-\alpha}$ takes $c_{00}$ onto itself it follows that $c_{00}$ is dense in $\tau^{\alpha}$.

Spaces $\tau^{\alpha}$ were introduced in \cite{GW} for $\alpha\in\N$. Their extensions to $\alpha>0$ have been defined in \cite[Th. 2.11]{ALMV} and \cite[Section 2]{A}, though with a slightly different presentation. Among other properties, these spaces satisfy the continuous inclusions
\begin{equation}\label{inclusion}
\tau^{\beta}\hookrightarrow\tau^{\alpha}
\hookrightarrow\ell^1,\ \text{ for }\beta>\alpha>0.
\end{equation}
Note also that $k^{\beta}\in \tau^{\alpha}$ if $\Real \beta<0$ or $\beta=0,$ for all $\alpha\geq 0$.

Moreover, the space $\tau^{\alpha}$ is a Banach algebra in the sense that there exists a (nonnecessarily equal to one) constant $M_\alpha$ such that
$\|f\ast g\|_{(\alpha)}\leq M_{\alpha}\|f\|_{(\alpha)}\|g\|_{(\alpha)}$
for $f,g\in\tau^{\alpha}$, where $\ast$ is the sequence convolution,
see \cite[Th. 2.11]{ALMV}. The description of the Gelfand transform of the Banach algebra $\tau^{\alpha}$ is quite simple:

For $f\in\tau^{\alpha},$ we consider the associated holomorphic function $\mathfrak{f}$ on the unit disc $\D$ (and continuous on $\overline{\D}$) given by
$\mathfrak{f}(z):=\sum_{n=0}^{\infty}f(n) z^n$.
Define
$A^{\alpha}(\D):=\{\mathfrak{f}\,:\, f\in\tau^{\alpha}\}$, 
endowed with pointwise multiplication and the norm $\|\mathfrak{f}\|_{A^{\alpha}(\D)}:=\|f\|_{(\alpha)}$. Thus
$A^{\alpha}(\D)$ and $\tau^{\alpha}$ are Banach algebras isometrically isomorphic.
The correspondence
$f\in\tau^{\alpha}\mapsto\mathfrak{f}\in A^{\alpha}(\D)$
is the Gelfand transform of $\tau^{\alpha}$, with range $A^{\alpha}(\D)$. It can be given in terms of Weyl differences.
Namely, for $\alpha\ge0$ and $n\in\N_0$, set
$$
\Delta^{-\alpha}\mathcal{Z}(n)
:=\sum_{j=0}^n k^{\alpha}(n-j)z^j, \quad z\in\overline{\D}.
$$
Clearly,
$\vert \Delta^{-\alpha}\mathcal{Z}(n)\vert\le\sum_{j=0}^nk^{\alpha}(n-j)\vert z\vert^j
\le\sum_{j=0}^n k^{\alpha}(n-j)=k^{\alpha+1}(n)$ uniformly on $\overline{\D}$. For estimates on compact subsets
$Q$ of $\D$, we have for $z\in Q$,
\begin{equation}\label{deltaestimate}
\vert \Delta^{-\alpha}\mathcal{Z}(n)\vert\le\sum_{j=0}^nk^{\alpha}(n-j)\vert z\vert^j
\le k^{\alpha}(n){1-\vert z\vert^{n+1}\over 1-\vert z\vert}\le M_Q k^{\alpha}(n)
\end{equation}
if $\alpha\ge1$ since $k^\alpha$ is increasing in this case, and
\begin{equation}\label{deltaestimate2}
\vert \Delta^{-\alpha}\mathcal{Z}(n)\vert\le\sum_{j=0}^nk^{\alpha}(n-j)\vert z\vert^j
\le k^{\alpha}(0){1-\vert z\vert^{n+1}\over 1-\vert z\vert}\le M_Q
\end{equation}
when $0<\alpha<1$, since $k^{\alpha}$ is decreasing now. Here $M_Q$ is a
constant depending on $Q$.

\medskip
Let $\mathfrak{f}(z)=\sum_{n=0}^{\infty}f(n) z^n$ be a holomorphic function
in $A^{\alpha}(\D)$. Using Fubini's theorem (for series) in the standard way,
it is readily seen that
\begin{equation}\label{int_rep}
\mathfrak{f}(z)
=\sum_{n=0}^{\infty}W^{\alpha}f(n)\Delta^{-\alpha}\mathcal{Z}(n), \quad z\in\overline{\D},
\end{equation}
where the series converges absolutely in $\overline\D$. Thus in particular
$\mathfrak{f}(1)=\sum_{n=0}^{\infty}W^{\alpha}f(n) k^{\alpha+1}(n)$.

\begin{example}\label{ex43}
\normalfont
Let $s\in \R$ and set
$\mathfrak{k}^{s}(z):=(1-z)^{-s}
=\sum_{n=0}^{\infty} k^{s}(n)z^n$, $z\in\D$.
Let $\alpha>0$. Clearly, $\mathfrak{k}^{s}\in A^{\alpha}(\D)$ if $ s\leq 0$ and
$\mathfrak{k}^{s}\notin A^{\alpha}(\D)$ for $s>0$, see \eqref{double2}. In particular,  for
$ s< 0$   one gets
\begin{equation}\label{identity1-1}
0=\mathfrak{k}^{s}(1)=\sum_{n=0}^{\infty}W^{\alpha}k^{s}(n) k^{\alpha+1}(n).
\end{equation}

Now, take $s\in (1,2)$ and $m\in\N.$ Then, by Example \ref{ex22}, \eqref{binom}, \eqref{sign} and \eqref{identity1-1}, one gets  
\begin{eqnarray*}
\|\mathfrak{k}^{-s}-\mathfrak{k}^{-1}\|_{A^{m}(\D)}
&=&\sum_{n=0}^{\infty}|W^{m}k^{-s}(n)-W^{m}k^{-1}(n)| k^{m+1}(n)\\
&=&\sum_{n=0}^{1}|k^{-s-m}(n+m)-k^{-1-m}(n+m)| k^{m+1}(n)-\sum_{n=0}^{1}W^{m}k^{-s}(n) k^{m+1}(n)\\
&\longrightarrow&
-\binom{m+1}{m}+(m+1)\binom{m+1}{m+1}=0, \hbox{ as } s\to 1^+.
\end{eqnarray*}

Moreover, by \eqref{inclusion} we have indeed
\begin{equation}\label{-sto-1}
\lim_{s\to 1^+}\|\mathfrak{k}^{-s}-\mathfrak{k}^{-1}\|_{A^{\alpha}(\D)}= 0 \hbox{  for all } \alpha\geq 0.
\end{equation}
\end{example}

\begin{remark}\label{integral_representation}
\normalfont
Representation \eqref{int_rep} entails some uniqueness properties. To begin with,
take $\alpha=m\in\N$. Then using the Cauchy formula for derivatives -which is to say, for the coefficients $f(n)$- one shows that
$$
W^m f(n)
={1\over 2\pi i}\int_{\vert \lambda\vert=r}
{(\lambda-1)^m\over \lambda^{m+n+1}}\mathfrak f(\lambda)\ d\lambda,\quad n\in\N_0.
$$

More generally, assume that $g$ is a sequence on $\N_0$ such that
$$
0=\sum_{n=0}^{\infty}g(n)\Delta^{-m}\mathcal{Z}(n)
$$
identically in $\D$. Then $g$ must be the null sequence. To see this, first note 
\begin{equation}\label{finitesum}
\Delta^{-m}\mathcal{Z}(n)
:=\sum_{j=0}^m k^m(n-j)z^j
={1\over(z-1)^m}\left(z^{m+n}+\sum_{j=1}^{m-1}P_{m,j}(n)z^{m-j}
+(-1)^m k^{m}(n)\right)
\end{equation}
for all $z\in\D$, where $P_{m,j}$ are polynomials of degree $m-1$ at most. This equality can be obtained by induction in $m$, on account of the identity
$\Delta^{-p}\mathcal{Z}(n)=\sum_{j=0}^n\Delta^{-(p-1)}\mathcal{Z}(j)$; $p,n\in\N$.

Evaluating at $z=0$ one gets $0=\sum_{n=0}^{\infty}g(n)k^{m+1}(n)$ and therefore
$$
0=\sum_{n=0}^{\infty}g(n)\left(z^{m+n}+\sum_{j=1}^{m-2}P_{m,j}(n)z^{m-j}+P_{m,m-1}(n)z\right)
$$
 for all $z\in\D$. Dividing by $z\not=0$ and then evaluating again the resulting polynomial at $z=0$, one obtains
$0=\sum_{n=0}^{\infty}g(n)P_{m,m-1}(n)$. By repetition of the argument we eventually arrive at
$\sum_{n=0}^{\infty}g(n)z^n=0$ ($z\in\D$), whence obviously $g(n)=0$.

\medskip
When $\alpha>0$ is arbitrary it is not clear to us whether or not the above  result on annihilation of the coefficients $g(n)$ holds true. However, it is still possible to obtain a weaker uniqueness property that is enough for the application we have in mind (see Theorem \ref{BAI}).

\begin{lemma}\label{Unicity}
For $\alpha>0$ and $n\in\N_0$ set $\omega_\alpha(n):=k^\alpha(n)$ if $\alpha\ge1$,
and
$\omega_\alpha(n):=1$ when $0<\alpha<1$.
Let $g$ be a sequence such that
$\sum_{n=0}^{\infty}\vert g(n)\vert
\omega_\alpha(n)<\infty$. Assume that
$$
\sum_{n=0}^{\infty}g(n)\Delta^{-\alpha}\mathcal{Z}(n)=0, \quad z\in\D.
$$
Then $g(n)=0$ for all $n\in\N_0$.
\end{lemma}

\begin{proof}
Set $M_{\alpha,g}:=\sum_{n=0}^\infty \omega_\alpha(n)\vert g(n)\vert$.
For $z\in\D$, we have
$$
\sum_{n=0}^\infty \vert g(n)\vert
\sum_{j=0}^n k^{\alpha}(n-j)\vert z\vert^j
=
\sum_{j=0}^\infty \sum_{n=j}^\infty k^{\alpha}(n-j)\vert g(n)\vert \vert z\vert^j
\le  {M_{\alpha,g}\over 1-\vert z\vert}<\infty,
$$
whence
$$
0=\sum_{n=0}^{\infty}g(n)\Delta^{-\alpha}\mathcal{Z}(n)
=\sum_{j=0}^{\infty}z^j\sum_{n=j}^{\infty}k^{\alpha}(n-j)g(n)
=\sum_{j=0}^{\infty}W^{-\alpha}g(j)z^j
$$
for every $z\in\D$, so that $W^{-\alpha}g(j)=0$ for all $j$. Now, $g\in\ell^1(k^\alpha)$ by hypothesis and therefore, 
by (\ref{deuvedoble}),
it follows  that $g=0$.
\end{proof}
\end{remark}

\medskip
A linear bounded operator $T\in{\mathcal B}(X)$ is $(C,\alpha)$-bounded if and only if 
there exists a Banach algebra bounded homomorphism  $\Theta_{\alpha}:\tau^{\alpha}\to \mathcal{B}(X)$, 
which furthermore is given by
$$
\Theta_{\alpha}(f)x=\sum_{n=0}^{\infty}W^{\alpha}f(n)\Delta^{-\alpha} \mathcal{T}(n)x,\quad x\in X,\,f\in \tau^{\alpha};
$$
see \cite[Theorem 3.5]{ALMV}. Neatly, such a homomorphism defines a functional calculus $\Phi_{\alpha}$ 
on functions of $A^{\alpha}(\D)$ given by
$\Phi_{\alpha}(\mathfrak{f}):=\Theta_{\alpha}(f)$, for $\mathfrak{f}(z)=\sum_{n=0}^\infty f(n)z^n$.
If $K_{\alpha}(T):=\sup_{n\in \N_0}\left\| M_T^{\alpha}(n) \right\|,$ then
\begin{equation}\label{ine}
\lVert\Phi_{\alpha}(\mathfrak{f})\rVert\leq K_{\alpha}(T)\lVert \mathfrak{f}\rVert_{A^{\alpha}(\D)}.
\end{equation}

We will denote the above functional calculus on $A^{\alpha}(\D)$ by $\Phi_\alpha(\mathfrak f)$ or
$\mathfrak{f}(T)$ indistinctly. 
It makes sense to apply (to arbitrary $\alpha)$ the regularization process, considered in \cite{HT} for $\alpha=0$, 
see \cite{Haase}.

\begin{definition}\label{a-reg}
\normalfont
Let $\alpha>0$ and let $T$ be a $(C,\alpha)$-bounded operator. We say that a function $\mathfrak{f}$ holomorphic in $\D$ is
$\alpha$-{\it regularizable}
if there is an element $\mathfrak{e}\in A^{\alpha}(\D)$ such that $\mathfrak{e}\mathfrak{f}\in A^{\alpha}(\D)$ and 
$\mathfrak{e}(T)$ is an injective operator. In such a case, put
$$
\mathfrak{f}(T):=\mathfrak{e}(T)^{-1}(\mathfrak{e}\mathfrak{f})(T).
$$
\end{definition}

The so-defined $\mathfrak{f}(T)$ does not depend on the $\alpha$-regularizer $\mathfrak{e}$ (see \cite[Lemma 1.2.1]{Haase}) and $\mathfrak{f}(T)$ is a closed operator.

Next, we connect the calculus $\Phi_{\alpha}$ with the calculus for sectorial operators, following \cite{HT}. Details for arbitrary $\alpha>0$ are included for the sake of completeness.

\begin{lemma}\label{Sector}
The operator $A:=I-T$ satisfies
$$
\left\| (\lambda-A)^{-1} \right\|
\leq K_{\alpha}(T)\frac{|\lambda|^{\alpha}}{(|\lambda-1|-1)^{\alpha+1}}
\leq K_{\alpha}(T)\frac{|\lambda|^{\alpha}}{|\Real \lambda|^{\alpha+1}},
\quad \Real \lambda<0.
$$
\end{lemma}

\begin{proof} We do notice that the spectral radius of $T$ is less than or equal to $1$, see
\cite[Lemma 1.1]{A}.
Let $\mu$ be a complex number such that $|\mu|>1$. By
\cite[Th. 4.4]{ALMV} and Example \ref{ex22} (i) one has
$$
(\mu-T)^{-1}
=\sum_{n=0}^{\infty}W^{\alpha}p_{\mu}(n)\Delta^{-\alpha}\mathcal{T}(n)
=\frac{(\mu-1)^{\alpha}}{\mu^{\alpha}}\sum_{n=0}^{\infty}\mu^{-n-1}\Delta^{-\alpha}\mathcal{T}(n).
$$
Then for $\Real\lambda<0$ we have
\begin{eqnarray*}
\lVert(\lambda-A)^{-1}\rVert&\leq&
K_{\alpha}(T)\frac{|\lambda|^{\alpha}}{|\lambda-1|^{\alpha}}
\sum_{n=0}^{\infty}|\lambda-1|^{-n-1}k^{\alpha+1}(n)\\
&=& K_{\alpha}(T)\frac{|\lambda|^{\alpha}}{(|\lambda-1|-1)^{\alpha+1}}
\leq K_{\alpha}(T)\frac{|\lambda|^{\alpha}}{|\Real\lambda|^{\alpha+1}},
\end{eqnarray*}
where we have used the identity \eqref{generating}.
\end{proof}

For $\theta\in (0,\pi)$ let $S_\theta$ denote the sector of angle $2\theta$ in the complex plane, which is symmetric with respect to the half-line $(0,\infty)$. Then
$|\lambda| / |\Real \lambda|\leq|\cos(\pi-\omega)|^{-1}$ for every
$\lambda\in \C\setminus \overline S_{\omega}$, with $\omega\in (\pi/2,\pi)$. From this,
it follows by Lemma \ref{Sector} that the operator $A=I-T$ is {\it sectorial} of angle $\pi/2$, that is,
the spectrum $\sigma(A)$ is contained in $\overline S_{\pi/2}$ and, for every
$\omega\in (\pi/2,\pi)$
and a constant $K_{\omega}$,
$$
\lVert \lambda(\lambda-A)^{-1}\rVert\leq K_{\omega},
\quad \lambda\in \C\setminus \overline S_{\omega}.
$$

It is well known that sectorial operators enjoy a remarkable functional calculus,
see \cite[Chapter 2]{Haase}. Namely,
for $\omega\in (\pi/2,\pi)$ let $\mathcal{E}_0(S_\omega)$ denote the space of holomorphic functions $\mathfrak{h}$ on
$S_\omega$ such that
$|\mathfrak{h}(z)|\leq K|z|^s$ for all $z\in S_\omega\cap \D$ for some constants $K,s>0$.
Take $\Lambda$ the oriented counterclockwise path given by
\begin{align*}
\Lambda := \{ r e^{i\varphi} : r\in[0, r_0] \} \cup \{r_0e^{i\psi} :
\psi\in(-\varphi,\varphi)\}\cup  \{ r e^{-i\varphi} : r\in[0, r_0] \},
\end{align*}
with $r_0>2$ and
$\pi/2<\varphi<\omega$.

Set $A=I-T$. Then the integral
\begin{equation}\label{hol1}
\mathfrak{h}(A)
=\frac{1}{2\pi i}\int_{\Lambda}\mathfrak{h}(\lambda)(\lambda-A)^{-1}\,d\lambda,
\quad \mathfrak{h}\in \mathcal{E}_0(S_\omega)
\end{equation}
makes sense, does not depend on the choice of $\varphi$ and defines a functional calculus on
$\mathcal{E}_0(S_\omega)$ (see \cite[p.46]{Haase}).
Moreover, when $\hbox{Ker}(I-T)=\{0\}$, one can extend the above calculus by regularization via powers of $A,$ see \cite[p.46]{Haase}.

We claim that the above sectorial functional calculus coincides, through the change of variable $z\mapsto(1-z)$, with the extended (by regularization)
 functional calculus $\Phi_{\alpha}$ on $A^{\alpha}(\D)$ introduced formerly. To prove that, it is enough by \cite[Prop. 1.2.7]{Haase} to check that both primary functional calculus, on
$\mathcal{E}_0(S_\varphi)$ and $A^{\alpha}(\D)$ respectively, coincide.

\begin{theorem}\label{Sectorial} Let $\omega\in (\pi/2,\pi)$,
$\mathfrak{h}\in \mathcal{E}_0(S_\omega)$ and $\mathfrak{f}(z):=\mathfrak{h}(1-z)$ for $z\in\D$. Then
$$
\mathfrak{f}\in A^{\alpha}(\D)\ \hbox{ and }\ \Phi_\alpha(\mathfrak{f})=\mathfrak{h}(A).
$$
\normalfont
\end{theorem}

\begin{proof}
Take $\Lambda$ as given in \eqref{hol1}. Note that for
$\lambda\in\Lambda\setminus \{0\}$ the function
$$
\mathfrak{p}_{1-\lambda}(z)
:=\frac{-1}{z-(1-\lambda)}
=\sum_{n=0}^{\infty}p_{1-\lambda}(n)z^n,
\quad z\in\D,
$$
lies in $A^{\alpha}(\D)$
with
$$
\lVert \mathfrak{p_{1-\lambda}}\rVert_{A^{\alpha}(\D)}
=\sum_{n=0}^{\infty}|W^{\alpha}p_{1-\lambda}(n)|k^{\alpha+1}(n)\leq \frac{|\lambda|^{\alpha}}{(|\lambda-1|-1)^{\alpha+1}},
$$
 see Example \ref{ex22} (i).

So, we have $\Phi_\alpha(p_{1-\lambda})=-(\lambda-A)^{-1}$, see \cite[Th. 4.4]{ALMV}.
Furthermore, the integral identity
$$
\mathfrak{f}(z)=\mathfrak{h}(1-z)
=\frac{-1}{2\pi i}
\int_{\Lambda}\mathfrak{h}(\lambda)\mathfrak{p}_{1-\lambda}(z)\,d\lambda
$$
holds in $A^{\alpha}(\D)$ since
$$
\left\| \frac{-1}{2\pi i}
\int_{\Lambda}\mathfrak{h}(\lambda)\mathfrak{p}_{1-\lambda}(z)\,d\lambda \right\|_{A^{\alpha}(\D)}\leq
\frac{K}{2\pi }
\int_{\Lambda}
|\lambda|^{s}\frac{|\lambda|^{\alpha}}
{(|\lambda-1|-1)^{\alpha+1}}\,d\lambda<\infty,
$$
where $K,s>0$ are such that $\vert\mathfrak h(\lambda)\vert\le K\vert\lambda\vert^s$,
$\lambda\in S_\omega\cap\D$.
Therefore, by \eqref{hol1} we conclude
$$
\Phi_\alpha(\mathfrak{f})
=\frac{-1}{2\pi i}
\int_{\Lambda}\mathfrak{h}(\lambda)\mathfrak{p}_{1-\lambda}(T)\,d\lambda
=\frac{1}{2\pi i}\int_{\Lambda}\mathfrak{h}(\lambda)(\lambda-A)^{-1}\,d\lambda
=\mathfrak{h}(A).
$$
\end{proof}

\begin{example}\label{ex23}
\normalfont
Let $s\in \R$ and let $\mathfrak{k}^{s}$ be given as in Example \ref{ex43}.
Assume that $T$ is $(C,\alpha)$-bounded such that
$\hbox{Ker}(I-T)=\{0\}$. 
Then $(I-T)^{-1}$ is injective and therefore functions  
$\mathfrak{e}\equiv (1-z)^n$, $n\in\N$, are candidate to be $\alpha$-regulariser. In fact, they are:
\begin{itemize}
\item[(i)]
If $n<s\leq n+1$ with $n\in\N_0$, then
$(1-z)^{n+1}\mathfrak{k}^{s}(z)
=\mathfrak{k}^{s-n-1}(z)\in A^{\alpha}(\D)$ and
$$
(I-T)^{-s}:=\mathfrak{k}^{s}(T)=(I-T)^{-n-1}\mathfrak{k}^{s-n-1}(T).
$$
\item[(ii)] Let us consider
$\log(1-z)=-\displaystyle\sum_{n=1}^{\infty}\frac{1}{n}z^n$ for $z\in\D$.
Then
$$
(\mathfrak{e}\mathfrak{f})(z):=(1-z)\log(1-z)
=-z+\sum_{n=1}^{\infty}\frac{z^{n+1}}{n(n+1)}\in A^{\alpha}(\D),
$$
so that $\log(1-z)$ is $\alpha$-regularizable by $1-z$, and then
$$
\log(I-T):=(I-T)^{-1}\left[(1-z)\log(1-z)\right](T)
$$
exists as a closed operator.
\end{itemize}
\end{example}

\medskip
Set, for $0<s<1$,
$
(I-T)^{s}:=\Phi_{\alpha}(\mathfrak{k}^{-s}).
$
The next property of the range of the above operator which will be used in Corollary \ref{Cor8.5}.

\begin{proposition}\label{prop4.8}
Let $T\in \mathcal{B}(X)$ be a $(C,\alpha)$-bounded operator with $\alpha\geq 0,$ and $0<s<1$.

Then
$$
\overline{(I-T)^{s}X}=\overline{(I-T)X}.
$$
\end{proposition}

\begin{proof}
Clearly, $(I-T)X\subseteq (I-T)^{s}X$ whence
$\overline{(I-T)X}\subseteq\overline{(I-T)^{s}X}$.

To the converse,
$$
\sum_{n=1}^{\infty}W^{\alpha}k^{-s}(n)k^{\alpha+1}(n)
=\sum_{n=0}^{\infty}W^{\alpha}k^{-s}(n)k^{\alpha+1}(n)-W^{\alpha}k^{-s}(0)
=-W^{\alpha}k^{-s}(0)
$$
since
$\sum_{n=0}^{\infty}W^{\alpha}k^{-s}(n)k^{\alpha+1}(n)$ equals $(z-1)^s$ at $z=1$.
Furthermore, taking $\beta=\alpha+1$ in the equality
$$
(I-T)\Delta^{-\beta}\mathcal{T}(n-1)=k^{\beta}(n)-\sum_{j=0}^{n}k^{\beta-\alpha-1}(n-j)\Delta^{-\alpha}\mathcal{T}(j)
$$
given in the proof Lemma \ref{dondeconverge}, one gets
$$
\Delta^{-(\alpha+1)}\mathcal{T}(n-1)(T-I)
=\Delta^{-\alpha}\mathcal{T}(n)-k^{\alpha+1}(n)I,\quad n\in\N.
$$

Therefore, for $x\in X$ and $y=(I-T)^{s}x$, one has
\begin{eqnarray*}
\displaystyle y&=&
\sum_{n=0}^{\infty}W^{\alpha}k^{-s}(n)\Delta^{-\alpha}\mathcal{T}(n)x
=W^{\alpha}k^{-s}(0)x+\sum_{n=1}^{\infty}W^{\alpha}k^{-s}(n)\Delta^{-\alpha}\mathcal{T}(n)x\\
&=&-\sum_{n=1}^{\infty}W^{\alpha}k^{-s}(n)(k^{\alpha+1}(n)I-\Delta^{-\alpha}\mathcal{T}(n))x
=\sum_{n=1}^{\infty}W^{\alpha}k^{-s}(n)\Delta^{-(\alpha+1)}\mathcal{T}(n-1)(T-I)x,
\end{eqnarray*}
so that $y\in\overline{(I-T)X}$. Then
$\overline{(I-T)^{s}X}\subseteq \overline{(I-T)X}$ as we wanted to show.
\end{proof}

\begin{remark}
\normalfont
For a $(C,\alpha)$ bounded operator $T$ as above it can be also shown that $(I-T)^{s}X$ is closed if and only if $(I-T)X$ is closed. 
We do not include the proof of this result since it is not needed in the paper. Such a property and Proposition \ref{prop4.8} 
are proved for power-bounded operators (case $\alpha=0$) in \cite[Prop. 2.1]{DL}.
\end{remark}

\section{Admissibility and fractional differences}\label{admisi}
\setcounter{theorem}{0}
\setcounter{equation}{0}

Here we generalize the notion of admissible function introduced in \cite{HT}, and prove key results about aproximation of the identity in 
the algebra $A^{\alpha}(\D)$. For this, we need some preliminary results.
The first one gives a representation of analytic functions by fractional differences $D^\alpha.$

\begin{lemma}\label{represen}
Let $\alpha>0$ and let $\mathfrak{f}(z)=\sum_{n=0}^{\infty}f(n)z^n$ be a holomorphic function in
$\D$ such that $D^{\alpha}f(j)\geq 0$ for all $j$ and $W^{-\alpha}(D^{\alpha}f)=f.$ Then $$
\mathfrak{f}(z)=\displaystyle\sum_{j=0}^{\infty}D^{\alpha}f(j)\Delta^{-\alpha}\mathcal{Z}(j),\quad z\in\D,
$$
and the series converges absolutely and uniformly in compact subsets of $\D$.
\end{lemma}

\begin{proof}
Let $z\in [0,1).$ By Fubini's Theorem one gets
$$
\sum_{j=0}^{\infty}D^{\alpha}f(j)\Delta^{-\alpha}\mathcal{Z}(j)
=\sum_{l=0}^{\infty}z^l\sum_{j=l}^{\infty}k^{\alpha}(j-l)D^{\alpha}f(j)=f(z).
$$
Furthermore, it is clear that the series converges uniformly and absolutely on compact subsets of
$\D$ since the representation holds in particular for $z\in[0,1)$ and $D^{\alpha}f\ge0$.
\end{proof}

\begin{proposition}\label{inversein}
Let $\mathfrak{f}$ be a free-zero holomorphic function on $\D$ extendible at $z=1$ with $\mathfrak{f}(1)\neq 0$
$(\mathfrak{f}(1)=\infty$ admitted\rm{)}. Let $\displaystyle\frac{1}{\mathfrak{f}}$ be given by
$\displaystyle\frac{1}{\mathfrak{f}(z)}=\sum_{n=0}^{\infty}g(n)z^n$, $z\in \D$,
and assume that there exists $W^{\alpha}g,$ with $g(0),W^{\alpha}g(0)\geq 0$ and $g(j),W^{\alpha}g(j)\leq 0,$ 
for $j\geq 1.$
Then $\displaystyle\frac{1}{\mathfrak{f}}\in A^{\alpha}(\mathbb{D})$, and
$$
\lVert\frac{1}{\mathfrak{f}}\rVert_{A^{\alpha}(\mathbb{D})}=2 W^{\alpha}g(0)-\frac{1}{\mathfrak{f}(1)},
$$
where $1/\infty:=0.$ Moreover, $\mathfrak{f}$ does not have zeros in $\overline{\mathbb{D}}$.
\end{proposition}

\begin{proof}
Since $-W^{\alpha}g$ is nonnegative it follows that there exists the finite sum
$$
\displaystyle\sum_{j=1}^{\infty}W^{\alpha}g(j)k^{\alpha+1}(j)
=\lim_{0<z\nearrow 1}\sum_{j=1}^{\infty}W^{\alpha}g(j)\Delta^{-\alpha}\mathcal{Z}(j)\\
\displaystyle=\lim_{0<z\nearrow 1}\sum_{j=1}^{\infty}g(j)z^j
=\lim_{0<z\nearrow 1}\frac{1}{\mathfrak{f}(z)}-W^{\alpha}g(0).
$$

Therefore $\frac{1}{\mathfrak{f}}\in A^{\alpha}(\mathbb{D})$, with
$\lVert\frac{1}{\mathfrak{f}}\rVert_{A^{\alpha}(\mathbb{D})}
=W^{\alpha}g(0)- \sum_{j=1}^{\infty}W^{\alpha}g(j)k^{\alpha+1}(j)
=2 W^{\alpha}g(0)-\frac{1}{\mathfrak{f}(1)}$.
In particular $\displaystyle\frac{1}{\mathfrak{f}}$ can be extended to $|z|=1$ and we conclude that
$\mathfrak{f}$ does not have zeros on
$\overline\D$.
\end{proof}

The formula given in Proposition \ref{Porfin} below is established in \cite[Lemma 2.7]{ALMV} for $f$ and $h$ in the
Banach algebra $\tau^\alpha$. Here we will need the formula for  $h\in\tau^\alpha$ but with $f$ not necessarily in 
$\tau^\alpha$, something which is available under certain compensatories  assumptions on $f$ and $h$. 
Its proof is rather involved and needs the following lemma.

\begin{lemma}\label{tech} Let $\alpha>0,$ $q\in\N$ and $h\in\N_0.$ Then
$$
k^{\alpha}(q+h)=-\sum_{p=0}^{q-1}k^{\alpha}(p)
\sum_{j=q-p}^{h+q-p}k^{-\alpha}(j)k^{\alpha}(h+q-p-j).
$$
\end{lemma}
\begin{proof}
We apply the induction method. First note that for $h=0$ and $q\in\N$,
$$
\displaystyle\sum_{p=0}^{q-1}k^{\alpha}(p)\sum_{j=q-p}^{q-p}k^{-\alpha}(j)k^{\alpha}(q-p-j)=\sum_{p=0}^{q-1}k^{\alpha}(p)k^{-\alpha}(q-p)
\displaystyle=-k^{\alpha}(q)
$$
Now, suppose that the identity of the statement is true for $h\in\N_{0}.$ Then, for $h+1,$ \begin{displaymath}\begin{array}{l}
\displaystyle\sum_{p=0}^{q-1}k^{\alpha}(p)\sum_{j=q-p}^{h+1+q-p}k^{-\alpha}(j)k^{\alpha}(h+1+q-p-j)=\sum_{p=0}^{q-1}k^{\alpha}(p)\biggl(k^{\alpha}(h+1)k^{-\alpha}(q-p)\\ \\
\displaystyle+\sum_{j=q-p+1}^{h+1+q-p}k^{-\alpha}(j)k^{\alpha}(h+1+q-p-j)\biggr)\\ \\
\displaystyle=-k^{-\alpha}(h+1)k^{\alpha}(q)+\sum_{p=0}^{q-1}k^{\alpha}(p)\sum_{j=q-p+1}^{h+1+q-p}k^{-\alpha}(j)k^{\alpha}(h+1+q-p-j)\\ \\
\displaystyle=k^{\alpha}(q)\sum_{j=1}^{h+1}k^{-\alpha}(j)k^{\alpha}(h+1-j)+\sum_{p=0}^{q-1}k^{\alpha}(p)\sum_{j=q-p+1}^{h+1+q-p}k^{-\alpha}(j)k^{\alpha}(h+1+q-p-j)\\ \\
\displaystyle=\sum_{p=0}^{q}k^{\alpha}(p)\sum_{j=q-p+1}^{h+1+q-p}k^{-\alpha}(j)k^{\alpha}(h+1+q-p-j)\\ \\
\displaystyle=-k^{\alpha}(h+(q+1))=-k^{\alpha}((h+1)+q),\quad q\in\N.
\end{array}\end{displaymath}
Thus we have completed the induction process and the proof is over.
\end{proof}

\begin{proposition}\label{Porfin} Let $\alpha>0$ and let $f,h$ be sequences such that
\begin{itemize}
\item[(i)] $f$ is a bounded sequence, $f(j)\ge0$, $D^\alpha f(j)\ge0$ for $j\in\N_0$ and
$W^{-\alpha}(D^{\alpha}f)=f$;
\item[(ii)] $h\in\tau^\alpha$ with $D^\beta h(0)\ge0$ and $D^\beta h(j)\le0$ ($j\ge1$),
for $\beta\in\{0,\alpha\}$;
\item[(iii)] $f*h \in\tau^\alpha$.
\end{itemize}

Then
$$
W^\alpha(f\ast h)(v):=\left(\sum_{j=0}^v\sum_{l=v-j}^v-\sum_{j=v+1}^\infty\sum_{l=v+1}^\infty\right)
k^\alpha(l+j-v)D^\alpha f(j)W^\alpha h(l),\ \hbox{ for } v\in\N_0.
$$
\end{proposition}

\begin{proof}
Note that
$|W^\alpha(f\ast h)(v)|=|D^\alpha(f\ast h)(v)|\leq \sum_{j=0}^{\infty}|k^{-\alpha}(j)|\sum_{l=0}^{j+v}f(j+v-l)|h(l)|<\infty$
since $f$ is bounded and $k^{-\alpha}$ and $h$ are in $\ell^1$.
Thus we can exchange summation order in $W^\alpha(f\ast h)(v)$ to find
\begin{eqnarray*}
W^\alpha(f\ast h)(v)&=&\sum_{l=0}^v h(l)\sum_{j=0}^{\infty} k^{-\alpha}(j)f(j+v-l)
+\sum_{l=v+1}^{\infty} h(l)\sum_{j=l-v}^{\infty} k^{-\alpha}(j)f(j+v-l)\\
&=&(h*D^{\alpha}f)(v)
+\sum_{l=v+1}^{\infty} W^{-\alpha}(W^{\alpha}h)(l)\sum_{j=l-v}^{\infty}k^{-\alpha}(j)W^{-\alpha}(D^{\alpha}f)(j+v-l).
\end{eqnarray*}
Furthermore, \begin{displaymath}\begin{array}{l}
\displaystyle|\sum_{l=v+1}^{\infty} W^{-\alpha}(W^{\alpha}h)(l)\sum_{j=l-v}^{\infty}k^{-\alpha}(j)W^{-\alpha}(D^{\alpha}f)(j+v-l)| \\ \\
\displaystyle\leq -\sum_{l=v+1}^{\infty}\sum_{p=l}^{\infty}k^{\alpha}(p-l)W^{\alpha}h(p)\sum_{j=l-v}^{\infty}|k^{-\alpha}(j)|\sum_{q=j+v-l}^{\infty}k^{\alpha}(q-j-v+l)D^{\alpha}f(q)\\ \\
\displaystyle=-\sum_{l=v+1}^{\infty}h(l)\sum_{j=l-v}^{\infty}|k^{-\alpha}(j)|f(j+v-l)<\infty.
\end{array}\end{displaymath}

Therefore, rearranging the above series and using Lemma \ref{tech} one obtains
 \begin{displaymath}\begin{array}{l}
\displaystyle\sum_{l=v+1}^{\infty} h(l)\sum_{j=l-v}^{\infty} k^{-\alpha}(j)f(j+v-l) \\ \\
\displaystyle=\sum_{q=0}^{\infty}\sum_{p=v+1}^{\infty}D^{\alpha}f(q)W^{\alpha}h(p)\sum_{l=v+1}^{p}k^{\alpha}(p-l)\sum_{j=l-v}^{q+l-v}k^{-\alpha}(j)k^{\alpha}(q-j-v+l)\\ \\
\displaystyle=\sum_{q=0}^{\infty}\sum_{p=v+1}^{\infty}D^{\alpha}f(q)W^{\alpha}h(p)\sum_{m=0}^{p-v-1}k^{\alpha}(m)\sum_{j=p-v-m}^{q+p-v-m}k^{-\alpha}(j)k^{\alpha}(q+p-v-j-m)\\ \\
\displaystyle=-\sum_{q=0}^{\infty}\sum_{p=v+1}^{\infty}D^{\alpha}f(q)W^{\alpha}h(p)k^{\alpha}(p-v+q).
\end{array}\end{displaymath}

On the other hand,
$$
(h*D^{\alpha}f)(v)=\sum_{j=0}^v D^\alpha f(j)h(v-j)=\sum_{j=0}^v D^\alpha f(j)\sum_{l=v-j}^\infty k^\alpha(l+j-v)W^\alpha h(l).
$$

Altogether, the result follows.
\end{proof}

Assumptions in the above results suggest the following definition.

\begin{definition}\label{a-admisi}
\normalfont
Let $\alpha\geq 0$ and let
$\mathfrak{f}(z)=\sum_{n=0}^{\infty}f(n)z^n$ be a holomorphic function on $\D$.
Then $\mathfrak{f}$ is said to be $\alpha$-{\it admissible} if:
\begin{itemize}
\item{(i)} The sequence $f$ is bounded,
$D^\beta f(n)\ge0$ and $W^{-\beta}(D^{\beta}f)=f$ for $\beta\in\{0,\alpha\}, n\in\N_0.$
\item{(ii)} The function $\mathfrak{f}$ does not have zeros in $\D$ and, if
$\displaystyle\frac{1}{\mathfrak{f}(z)}=\sum_{n=0}^{\infty}g(n)z^n$, $z\in \D$,
then the differences $W^{\beta}g$ exist and satisfy $W^{\beta}g(n)\leq 0$ for
 $n\geq 1$, and $W^{\beta}g(0)\geq 0,$ for $\beta\in\{0,\alpha\}.$
\end{itemize}
\end{definition}

\begin{remark}\normalfont
If $\alpha=m\in\N,$ then condition $W^{-m}(D^{m}f)=f$ in the $m$-admissibility is redundant.
\end{remark}

Next, we discuss $\alpha$-admissibility in relation with the algebra
$A^\alpha(\D)$.
Let $\mathfrak f(z)=\sum_{j=0}^\infty f(j)z^j$ be an $\alpha$-admissible function on
$\D$ with inverse
$\displaystyle\frac{1}{\mathfrak{f}(z)}=\sum_{j=0}^\infty g(j)z^j$. The fact that
$\displaystyle\mathfrak f \frac{1}{\mathfrak{f}}=1$ means that $f\ast g=k^0(=\delta_0)$.
For every $n\in\N$,
let $\mathfrak{g}_{n}$ be the function
$$
\mathfrak{g}_{n}(z)
=\frac{1}{\mathfrak{f}(z)}
\sum_{j=0}^{n-1}D^{\alpha}f(j)\Delta^{-\alpha}\mathcal{Z}(j)=\sum_{j=0}^{\infty}g_{n}(j)z^j,
$$
and put $\mathfrak{f}_n:=\mathfrak{g}_{n}\mathfrak{f}$. Note that $\mathfrak{g}_{n}\in A^{\alpha}(\mathbb{D})$
since $\displaystyle\frac{1}{\mathfrak{f}},\mathfrak{f}_n\in A^{\alpha}(\mathbb{D}$). Note also that
$\mathfrak{g}_{n}(1)=0$ if $\mathfrak{f}(1)=\infty$, and that $\mathfrak{g}_{n}(1)\le1$ if
$\mathfrak{f}(1)<\infty,$ by Lemma \ref{represen}. We need to find a suitable estimate of the norm of 
$\mathfrak{g}_{n}$ in $A^{\alpha}(\mathbb{D})$, and this requires a careful analysis of the sign of certain coefficients, which can be done thanks to the key formula established in Proposition \ref{Porfin}.

\begin{theorem}\label{lemma6.1}
Let $\mathfrak f$ be an $\alpha$-admissible function on $\D$ and $\mathfrak{g}_n$ as above.
Then
$$
\lVert \mathfrak{g}_{n}\rVert_{A^{\alpha}(\mathbb{D})}\leq 2-\mathfrak{g}_{n}(1)\quad
\hbox{ for all } n\geq 1.
$$
\end{theorem}

\begin{proof}
By Lemma \ref{represen} one has
$\mathfrak{f}(z)=\sum_{j=0}^\infty D^{\alpha}f(j)\Delta^{-\alpha}\mathcal{Z}(j)$ for all $z\in\D.$ This series converges absolutely and uniformly on compacts subsets. Set
$$
{\frak F}_n(z)
:=\sum_{j=n}^\infty D^\alpha f(j)\Delta^{-\alpha}\mathcal{Z}(j).
$$
It is readily seen that the sequence of coefficients $f_n$ of $\mathfrak{f}_n=\mathfrak{g}_{n}\mathfrak{f}$ is such that
$$
D^\alpha f_n(j)=D^\alpha f(j), \hbox{ if } 0\le j<n;\quad
D^\alpha f_n(j)=0, \hbox{ if } j\ge n.
$$
Therefore the sequence of coefficients $\varphi_n$ of ${\frak F}_n$ satisfies
$$
D^\alpha \varphi_n(j)=0,\ \hbox{ if } 0\le j<n;\,
D^\alpha \varphi_n(j)=D^\alpha f(j),\ \hbox{ if } j\ge n,
$$
since $D^{\alpha}\varphi_n=D^{\alpha}f-D^{\alpha}f_n.$ Thus applying Proposition \ref{Porfin} to the polynomial $\mathfrak f_n$ we have
$$
W^\alpha g_n(v)=W^\alpha(f_n\ast g)(v):=\sum_{j=0}^{n-1}\sum_{l=v-j}^v
k^\alpha(l+j-v)D^\alpha f(j)W^\alpha g(l)\le0
$$
for every $v\ge n$. Furthermore $$W^\alpha g_n(0)=D^\alpha f(0)W^\alpha g(0)-\sum_{j=1}^{n-1}\sum_{l=1}^\infty
k^\alpha(l+j)D^\alpha f(j)W^\alpha g(l)\geq 0.$$

On the other hand $\varphi_n\ast g=f\ast g-f_n\ast g=\delta_0-f_n\ast g\in\tau^\alpha$ whence applying
Proposition \ref{Porfin} to $\varphi_n\ast g$,
\begin{eqnarray*}
W^\alpha g_n(v)&=&W^\alpha\delta_0(v)-W^\alpha(\varphi_n\ast g)(v)=
\delta_0(v)-W^\alpha(\varphi_n\ast g)(v)\\
&=&\sum_{j=n}^\infty\sum_{l=v+1}^\infty
k^\alpha(l+j-v)D^\alpha f(j)W^\alpha g(l)\le0
\end{eqnarray*}
for every $v$ such that $1\le v\le n-1$. In addition, $$W^\alpha g_n(0)=1+\sum_{j=n}^\infty\sum_{l=1}^\infty
k^\alpha(l+j)D^\alpha f(j)W^\alpha g(l)\leq 1.$$

All in all,
\begin{eqnarray*}
\lVert\mathfrak{g}_{n}\rVert_{A^{\alpha}(\mathbb{D})}&=&\sum_{j=0}^{\infty}|W^{\alpha}g_{n}(j)|k^{\alpha+1}(j)=W^{\alpha}g_{n}(0)-\sum_{j=1}^{\infty}W^{\alpha}g_{n}(j)k^{\alpha+1}(j)\\
&=&2W^{\alpha}g_{n}(0)-\sum_{j=0}^{\infty}W^{\alpha}g_{n}(j)k^{\alpha+1}(j)=2W^{\alpha}g_{n}(0)-\mathfrak{g}_{n}(1)
\leq 2-\mathfrak{g}_{n}(1),
\end{eqnarray*}
as we wanted to proof.
\end{proof}

Analytic polynomials are dense in $A^\alpha(\D)$, which implies that the set of polynomials vanishing at $z=1$ is dense in the closed ideal of $A^\alpha(\D)$ formed by the functions which are zero at $z=1$. Theorem \ref{BAI} 
shows in particular that the sequence 
of functions $(\mathfrak{g}_{n})_{n\ge1}$ is a bounded approximate identity for that ideal. Even more, if
$\mathfrak{f}(1)<\infty$, that family converges in norm to the identity element in the algebra $A^\alpha(\D)$.

\begin{theorem}\label{BAI}
Let $\mathfrak{f}$ be an $\alpha$-admissible function and set $\mathfrak{g}_{n}$ as above.
\begin{itemize}
\item[(i)] If $\mathfrak{f}(1)<\infty$ then $\lim_{n\to\infty}\mathfrak{g}_{n}(z)=1 \ \text{in }A^{\alpha}(\D)$.
\item[(ii)] If $\mathfrak{f}(1)=\infty$ then
$\Vert \mathfrak{g}_n\Vert_{A^{\alpha}(\D)}\le2$ for every $n$.
\item[(iii)]  If $(1-z)\mathfrak{f}(z)\in A^{\alpha}(\D)$ and $D^{\alpha}f(j)j^{\alpha}\to 0$ as $j\to\infty,$ then
$$
\lim_{n\to\infty}(1-z)\mathfrak{g}_{n}(z)=1-z\ \text{in }A^{\alpha}(\D).
$$
\end{itemize}
\end{theorem}

\begin{proof}
(i) If $\mathfrak{f}(1)<\infty,$ by Lemma \ref{represen} we have $$\mathfrak{f}(1)=\lim_{0<z\nearrow 1}\sum_{j=0}^{\infty}D^{\alpha}f(j)\Delta^{-\alpha}\mathcal{Z}(j)=\sum_{j=0}^{\infty}D^{\alpha}f(j)k^{\alpha+1}(j).$$ Therefore $$\lVert \mathfrak{g}_{n}-1 \rVert_{A^{\alpha}(\D)}=\frac{1}{\mathfrak{f}(1)}\sum_{j=n}^{\infty}D^{\alpha}f(j)k^{\alpha+1}(j)\to 0,\quad n\to\infty.$$

(ii) If $\mathfrak{f}(1)=\infty,$ then the proof of Theorem \ref{lemma6.1} gives $$\Vert \mathfrak{g}_n\Vert_{A^{\alpha}(\D)}=2W^{\alpha}g_{n}(0)-\mathfrak{g}_{n}(1)=2W^{\alpha}g_{n}(0)-\frac{1}{\mathfrak{f}(1)}\sum_{j=0}^{n-1}D^{\alpha}f(j)k^{\alpha+1}(j)\leq 2.$$

(iii) Note that
$$
(1-z)\mathfrak{g}_{n}(z)=\frac{1}{\mathfrak{f}(z)}
\biggl( \sum_{j=0}^{n-1}D^{\alpha}f(j)\Delta^{-\alpha}\mathcal{Z}(j)
-\sum_{j=0}^{n-1}D^{\alpha}f(j)z\Delta^{-\alpha}\mathcal{Z}(j) \biggr).
$$
with
\begin{eqnarray*}
\displaystyle\sum_{j=0}^{n-1}D^{\alpha}f(j)z\Delta^{-\alpha}\mathcal{Z}(j)&=&\sum_{j=1}^{n}D^{\alpha}f(j-1)\sum_{l=1}^{j}k^{\alpha}(j-l)z^{l}\\
&=&\sum_{j=1}^{n}D^{\alpha}f(j-1)(\Delta^{-\alpha}\mathcal{Z}(j)-k^{\alpha}(j)).
\end{eqnarray*}

Put
\begin{equation}\label{4.2}
(1-z)\mathfrak{g}_{n}(z)=\mathfrak{h}_n(z)-\mathfrak{r}_{n}(z)
\end{equation}
where
\begin{eqnarray*}
\mathfrak{h}_n(z)&:=&\frac{1}{\mathfrak{f}(z)}\biggl( D^{\alpha}f(0)+\displaystyle\sum_{j=1}^{n-1}(D^{\alpha}f(j)-D^{\alpha}f(j-1))\Delta^{-\alpha}\mathcal{Z}(j)+\sum_{j=1}^{n}D^{\alpha}f(j-1)k^{\alpha}(j)\biggr)\\
&=&\frac{1}{\mathfrak{f}(z)}\biggl(
D^{\alpha}f(0)+\sum_{j=1}^{n}D^{\alpha}f(j-1)k^{\alpha}(j)-\displaystyle\sum_{j=1}^{n-1}D^{\alpha+1}f(j-1)\Delta^{-\alpha}\mathcal{Z}(j)\biggr)
\end{eqnarray*}
and
$\displaystyle\mathfrak{r}_n(z):=\frac{1}{\mathfrak{f}(z)}D^{\alpha}f(n-1)\Delta^{-\alpha}\mathcal{Z}(n)$.

Let us remark that both $\mathfrak{h}_n$ and $\mathfrak{r}_n$
belong to $A^{\alpha}(\D)$. We claim that
\begin{equation}\label{hache}
\lim_{n\to\infty}\Vert \mathfrak{h}_{n}-(1-z) \Vert_{A^{\alpha}(\D)}=0 \quad
\hbox{ and } \quad
\lim_{n\to\infty}\Vert\mathfrak{r}_n\Vert _{A^{\alpha}(\D)}=0.
\end{equation}
To see this, note that
$$
\Vert \mathfrak{h}_{n}-(1-z) \Vert_{A^{\alpha}(\D)}\leq M_\alpha  \lVert\frac{1}{\mathfrak{f}}\rVert_{A^{\alpha}(\mathbb{D})}\lVert \mathfrak{h}_{n}\mathfrak{f}-(1-z)\mathfrak{f}\rVert_{A^{\alpha}(\mathbb{D})},
$$

By Lemma \ref{represen},
\begin{eqnarray*}
(1-z)\mathfrak{f}(z)&=&\displaystyle\sum_{j=0}^{\infty}D^{\alpha}f(j)\Delta^{-\alpha}\mathcal{Z}(j)
-\sum_{j=0}^{\infty}D^{\alpha}f(j)z\Delta^{-\alpha}\mathcal{Z}(j)\\
&=&D^{\alpha}f(0)+\displaystyle\sum_{j=1}^{\infty}(D^{\alpha}f(j)-D^{\alpha}f(j-1))\Delta^{-\alpha}\mathcal{Z}(j)
+\displaystyle\sum_{j=1}^{\infty}D^{\alpha}f(j-1)k^{\alpha}(j)\\
&=&D^{\alpha}f(0)+\displaystyle\sum_{j=1}^{\infty}D^{\alpha}f(j-1)k^{\alpha}(j)
-\displaystyle\sum_{j=1}^{\infty}D^{\alpha+1}f(j-1)\Delta^{-\alpha}\mathcal{Z}(j),
\end{eqnarray*}
where the series in the latter line converges since
$$
\sum_{j=1}^{\infty}D^{\alpha}f(j-1)k^{\alpha}(j)
\le K\sum_{j=0}^{\infty}D^{\alpha}f(j)k^{\alpha}(j)=K\mathfrak{f}(0),
$$
for some constant $K>0.$
Moreover, the expansion
$$
(1-z)\mathfrak{f}(z)=D^{\alpha}f(0)
+\displaystyle\sum_{j=1}^{\infty}D^{\alpha}f(j-1)k^{\alpha}(j)
-\displaystyle\sum_{j=1}^{\infty}D^{\alpha+1}f(j-1)\Delta^{-\alpha}\mathcal{Z}(j)
$$
shows that the series gives us the representation (\ref{int_rep}) for
$(1-z)\mathfrak{f}$ as an element of $A^{\alpha}(\D)$ by
\eqref{deltaestimate}, \eqref{deltaestimate2} and Lemma \ref{Unicity}.
Hence,
\begin{eqnarray*}
\lVert \mathfrak{h}_{n}\mathfrak{f}-(1-z)\mathfrak{f}\rVert_{A^{\alpha}(\mathbb{D})}&=&\lVert
-\displaystyle\sum_{j=n+1}^{\infty}D^{\alpha}f(j-1)k^{\alpha}(j)
+\displaystyle\sum_{j=n}^{\infty}D^{\alpha+1}f(j-1)\Delta^{-\alpha}\mathcal{Z}(j)\rVert_{A^{\alpha}(\mathbb{D})}\\
&=&\displaystyle\sum_{j=n+1}^{\infty}D^{\alpha}f(j-1)k^{\alpha}(j)+
\displaystyle
\sum_{j=n}^{\infty}|D^{\alpha+1}f(j-1)|k^{\alpha+1}(j),
\end{eqnarray*}
whence
$\lVert \mathfrak{h}_{n}-(1-z) \rVert_{A^{\alpha}(\D)}\to 0$ as $n\to\infty.$

On the other hand,
\begin{eqnarray*}
\lVert \mathfrak{r}_n\rVert_{A^{\alpha}(\D)}
&\leq& M_\alpha \Vert 1/\mathfrak{f}\rVert_{A^{\alpha}(\D)}
\lVert D^{\alpha}f(n-1) \Delta^{-\alpha}{\mathcal Z}(n)\rVert_{A^{\alpha}(\D)}\\
&=&M_\alpha \Vert 1/\mathfrak{f}\rVert_{A^{\alpha}(\D)}
 D^{\alpha}f(n-1) k^{\alpha+1}(n)\to 0,\quad \text{ as }n\to\infty,
\end{eqnarray*}
by the assumption in part (iii). Thus the proof is over.
\end{proof}

In the following section we provide examples of $\alpha$-admissible functions satisfying the conditions assumed formerly.

\section{Examples}\label{ejemplos}\label{exadmisi}
\setcounter{theorem}{0}
\setcounter{equation}{0}

In \cite {HT}, examples of $0$-admissible functions are given on the basis of a classical theorem of Kaluza involving decreasing  log-convex sequences
(and a shorter, new, proof of such a theorem is also given in \cite[Prop. 4.4]{HT}).
Here we present a generalized version of Kaluza's theorem to  log-convex sequences of higher order.
Though this result will not be applied until the occurrence of our concrete second example of $\alpha$-admissible function, it is stated right now because it provides a set of examples of a general character.

\medskip
{\bf Log-convex sequences of higher difference order}.
We say that  a non-identically zero sequence $f$
is logarithmically convex, or log-convex for short, of degree $m\in\N_0$ if for every $p\in\{0,1,\dots,m\}$
and $j\ge0$ one has
$$
D^p f(j)>0\ \hbox{ and }\ \left(D^p f(j+1)\right)^2\le D^p f(j) D^p f(j+2),\ \forall j\ge0.
$$
In other words, the sequence
$(D^p f(j+1)/D^p f(j))_{j=0}^\infty$ is increasing for all $p\in\{0,1,\dots,m\}$.

\begin{remark}\label{remark7.1}\normalfont If a decreasing sequence $f$ satisfies that $D^p f(j)>0,$ for $j\ge0$ and 
$0\leq p\leq m$ with $p\in\N_0,$ then $D^\beta f(j)>0$ for $j\ge0$  and $0\leq\beta\leq m$ with $\beta$ a real number. Indeed,  if $p-1<\beta<p$ with $p\in\N$ and $p\leq m,$  by Fubini's Theorem we have 
$D^\beta f(j)=D^{\beta-(p-1)}D^{p-1}f(j)=\sum_{l=0}^{\infty}k^{p-1-\beta}(l)D^{p-1}f(l+j) 
=\sum_{l=1}^{\infty}k^{p-1-\beta}(l)(D^{p-1}f(l+j)-D^{p-1}f(j))>0$,
since  $D^{p}f>0$ on $\N_0$ entails $D^{p-1}f$ decreasing on $\N_0$.
\end{remark}

The following is the generalization of Kaluza's theorem mentioned before.

\begin{theorem}\label{logconvex}
Let $\frak f$ be a holomorphic function on the unit disc $\D$ of Taylor series $\sum_{j=0}^\infty f(j)z^j$ such that the sequence $f:=(f(j))_{j=0}^\infty$ is decreasing, log-convex of degree $m$, with $W^{-\beta}(D^{\beta}f)=f$ and 
$D^{\beta}f(j)j^{\beta}\to 0$ $(j\to\infty)$ for $0\leq \beta\leq m$. Then $\frak f$ is zero-free on $\D$ and,
if ${1\over \frak f(z)}=\sum_{j=0}^\infty g(j)z^n$, one has
$$
W^\beta g(0)>0 \quad \hbox{ and }\quad  W^\beta g(j)<0,\quad j\ge1,\  0\le \beta\le m.
$$
That is, $\frak f$ is $\beta$-admissible for each $0\leq \beta\leq m$, and therefore in particular 
$\displaystyle{1\over \frak f(z)}$ belongs to $A^m(\D)$.
\end{theorem}

The proof of Theorem \ref{logconvex}  is long, and we prefer to include it at the end of this paper as an appendix.
We now proceed with specific examples.

\begin{example}\label{Ex63}
\normalfont
Let $0<s<1$.
The function $\mathfrak{k}^{s}(z)=(1-z)^{-s}$ is
$\alpha$-admissible for all $\alpha\geq 0$.
Indeed, the sequence of its coefficients $k^{s}$ is bounded, and by Example \ref{ex22} (ii),
$$
D^{\alpha} k^s(n)=\frac{\Gamma(1-s+\alpha)\Gamma(s+n)}{\Gamma(s)\Gamma(1-s)\Gamma(n+\alpha+1)}>0,
\quad n\in\N_0.
$$

Moreover, for every $n\in\N_0$, by \eqref{generating} we have
\begin{eqnarray*}
W^{-\alpha}D^{\alpha}k^{s}(n)
&=&\frac{1}{\Gamma(s)\Gamma(1-s)}\sum_{l=0}^{\infty}k^{\alpha}(l)\int_{0}^1(1-x)^{s+n+l+1}x^{\alpha-s}\, dx\\
&=&\frac{1}{\Gamma(s)\Gamma(1-s)}\int_{0}^1(1-x)^{s+n+1}x^{-s}\, dx
=\frac{\Gamma(s+n)}{\Gamma(s)n!}=k^{s}(n).
\end{eqnarray*}

Notice that $\mathfrak{k}^{s}$ does not have zeros on $\D$ --in fact 
$1/\mathfrak{k}^{s}=\mathfrak{k}^{-s}\in A^{\alpha}(\D)$,
see \eqref{double2}--, and by \cite[Lemma 1.1]{AM}
$$
W^{\alpha}k^{-s}(n)=D^{\alpha}k^{-s}(n)
=\frac{\Gamma(1+s+\alpha)\Gamma(-s+n)}{\Gamma(-s)\Gamma(1+s)\Gamma(n+\alpha+1)}<0,\quad n\in \N.
$$
Also,
\begin{eqnarray*}
W^{\alpha}k^{-s}(0)&=&\sum_{l=0}^{\infty}k^{-\alpha}(l)k^{-s}(l)
=1+\sum_{l=1}^{\infty}k^{-\alpha}(l)(k^{1-s}(l)-k^{1-s}(l-1))\\
&=&\sum_{l=0}^{\infty}k^{-\alpha}(l)k^{1-s}(l)-\sum_{l=0}^{\infty}\frac{(l-\alpha)}{l+1}k^{-\alpha}(l)k^{1-s}(l)\\
&=&\frac{(1+\alpha)}{\Gamma(1-s)\Gamma(1+s)}\int_0^1 x^{s+\alpha}(1-x)^{-s}\,dx
=\frac{\Gamma(s+\alpha+1)}{\Gamma(1+s)\Gamma(1+\alpha)}>0.
\end{eqnarray*}

Furthermore, $(1-z)\mathfrak{k}^{s}(z)=\mathfrak{k}^{s-1}(z)$ belongs to the algebra $A^{\alpha}(\D)$
and $D^{\alpha}k^{s}(n)n^{\alpha}\to 0$ as $n\to \infty$ (Example \ref{ex22} (ii)).
Therefore, functions $\mathfrak{k}^{s}$ satisfy the conditions of Theorem \ref{BAI}.
\end{example}

\begin{example}\label{Ex64}
\normalfont
The holomorphic function 
$\displaystyle\mathfrak{L}(z)=\frac{-\log(1-z)}{z}=\sum_{j=0}^{\infty}\frac{z^j}{j+1}
=\sum_{j=0}^{\infty}L(j)z^j$, $z\in\D$,
is log-convex, and $0$-admissible as pointed out in \cite[p. 283]{HT}.
By Example \ref{ex22} (iii), one has $D^{m}L(j)=\displaystyle\frac{m!}{(j+1)\ldots(j+m+1)}$, whence 
$(D^{m}L(j+1))^2<D^{m}L(j)D^{m}L(j+2)$ for every $m\in\N$.
Moreover, for $\alpha\geq 0,$ $D^{\alpha}L(j)j^{\alpha}\to 0$ as $j\to\infty$ (see Example \ref{ex22} (iii)), 
and by \eqref{generating} we have 
$$
W^{-\alpha}D^{\alpha}L(j)
=
\sum_{l=0}^{\infty}k^{\alpha}(l)\int_{0}^1(1-x)^{j+l}x^{\alpha}\, dx
=\int_{0}^1(1-x)^{j}\,dx=L(j).
$$
Then by Theorem \ref{logconvex}, the function $\mathfrak{L}$ is $\alpha$-admissible for all $\alpha\geq 0.$

Furthermore,
$(1-z)\mathfrak{L}(z)\in A^{\alpha}(\D)$,
which, along with $\lim_{j} D^{\alpha}L(j)j^{\alpha}\to 0$, impliy that $\mathfrak{f}$ satisfies the conditions of Theorem \ref{BAI}.
\end{example}

\section{Approximating identities with Taylor coefficients}\label{taylor}
\setcounter{theorem}{0}
\setcounter{equation}{0}

The construction of approximate identities in $A^\alpha(\D)$ carried out in Theorem \ref{BAI}
requires using partial sums formed with fractional differences. Such approximate identities are suitable to our objectives in Section \ref{DOF} and Section \ref{PoiLo} below, concerning general domains of operatorial functions or higher degree $\alpha$.
For the specific examples given in Section \ref{ejemplos} and $\alpha\in(0,1)$ one can obtain approximate identities from partial sums of Taylor expansions.

Let $\mathfrak{f}$ be a zero-free holomorphic function on the unit disc with 
$(1/\mathfrak{f})(z)=\sum_{j=0}^\infty g(j)z^j$.
Let then $\mathfrak{g}^0_n$ denote the approximating
function given in \cite[Lemma 4.6]{HT}, that is,
$\mathfrak{g}^0_n(z):=(1/\mathfrak{f})(z)\sum_{j=0}^{n-1} f(j)z^j$, $|z|\leq 1$.
Put $\mathfrak{f}^{0}_n(z):=\mathfrak{g}^{0}_{n}(z)\mathfrak{f}(z)$ and $g_{n}^0(j)=\sum_{l=0}^{n-1}f(l)g(j-l)$.
By \cite[Lemma 4.6]{HT} we have
\begin{equation}\label{eq6.1}\mathfrak{g}^0_n(z)=1+\sum_{j=n}^{\infty}g_n^0(j)z^j,\end{equation}

Theorem \ref{approximation01} and
Theorem \ref{approximation02} below prove, respectively for the $\alpha$-admissible functions
$\mathfrak{k}^{s}$ and $\mathfrak{L}$ of Section \ref{ejemplos}, and with $\alpha$ in a certain range of values, that
the sequence $(\mathfrak{g}^{0}_{n})_{n\ge1}$ is a bounded approximate identity for the closed ideal of
$A^\alpha(\D)$ of functions vanishing at $z=1$.

\begin{theorem}\label{approximation01}
Take $s$ such that $0<s<1$. Let $\mathfrak{f}=\mathfrak{k}^{s}$ and let $\mathfrak{g}^0_{n}$ be as above, 
corresponding to
$\mathfrak{f}=\mathfrak{k}^{s}$.
Then for $0\leq \alpha<1-s$ one has
\begin{itemize}
\item[(i)]  $\Vert \mathfrak{g}^{0}_n\Vert_{A^{\alpha}(\D)}\le M$ for every $n$, with $M>0$ independent of $n.$
\item[(ii)] $\lim_{n\to\infty}(1-z)\mathfrak{g}^0_{n}(z)=1-z\ \text{in }A^{\alpha}(\D).$
\end{itemize}
\end{theorem}
\begin{proof}
(i) First of all, note that $g_{n}^0(j)=\sum_{l=0}^{n-1}k^s(l)k^{-s}(j-l)<0,$ for $j\geq n.$ Then by \eqref{eq6.1} we have
$D^{\alpha}g_{n}^0(0)=1+\sum_{l=n}^{\infty}k^{-\alpha}(l)g_{n}^0(l)>0$, 
and 
$D^{\alpha}g_{n}^0(j)=\sum_{l=n}^{\infty}k^{-\alpha}(l-j)g_{n}^0(l)>0$ for $1\leq j\leq n-1$.
Also, if $j\geq n$, by \eqref{eq6.1} and the $\alpha$-admisibility of $\mathfrak{k}^s$ we obtain
$D^{\alpha}g_{n}^0(j)
=\sum_{u=0}^{n-1}k^s(u)\sum_{l=j}^{\infty}k^{-\alpha}(l-j)k^{-s}(j-u)
=\sum_{u=0}^{n-1}k^s(u)D^{\alpha}k^{-s}(j-u)<0.$
Secondly, by \eqref{int_rep} one gets 
$\mathfrak{g}^{0}_n(1)=\displaystyle\sum_{j=0}^{\infty}D^{\alpha}g_{n}^0(j)k^{\alpha+1}(j)=\frac{1}{\mathfrak{k^{s}}(1)}\sum_{j=0}^{n-1} k^{s}(j)=0$  since $\mathfrak{g}^{0}_n\in A^{\alpha}(\D)$.

Therefore,
\begin{eqnarray*}
\Vert \mathfrak{g}^{0}_n\Vert_{A^{\alpha}(\D)}&=&\displaystyle\sum_{j=0}^{\infty}|D^{\alpha}g_{n}^0(j)|k^{\alpha+1}(j)
=\displaystyle\sum_{j=0}^{n-1}D^{\alpha}g_{n}^0(j)k^{\alpha+1}(j)-\sum_{j=n}^{\infty}D^{\alpha}g_{n}^0(j)k^{\alpha+1}(j)\\
&=&2\displaystyle\sum_{j=0}^{n-1}D^{\alpha}g_{n}^0(j)k^{\alpha+1}(j)-\mathfrak{g}^{0}_n(1)
=2+2\displaystyle\sum_{j=0}^{n-1}k^{\alpha+1}(j)\sum_{l=n}^{\infty}k^{-\alpha}(l-j)g_{n}^0(l)\\
&=&2-2\displaystyle\sum_{j=0}^{n-1}k^{\alpha+1}(j)\sum_{l=n}^{\infty}k^{-\alpha}(l-j)\sum_{u=n}^{l}k^s(u)k^{-s}(l-u)\\
&=&2-2\displaystyle\sum_{j=0}^{n-1}k^{\alpha+1}(j)\sum_{u=n}^{\infty}k^s(u)D^{s}k^{-\alpha}(j-u)\\
&=&2-2\frac{\Gamma(1+\alpha+s)}{\Gamma(1+\alpha)\Gamma(-\alpha)}\displaystyle\sum_{u=n}^{\infty}k^s(u)\sum_{j=0}^{n-1}k^{\alpha+1}(j)\frac{\Gamma(-\alpha+u-j)}{\Gamma(s+1+u-j)}.
\end{eqnarray*}

It is readily seen that \cite[Lemma 1.2]{AM} and \cite[Th. 1.3]{AM} are valid for values of $\alpha<0$.
Therefore, applying  \cite[Th. 1.3]{AM} we have, for $u\ge n$,
$$
k^s(u)\sum_{j=0}^{n-1}k^{\alpha+1}(j)\frac{\Gamma(-\alpha+u-j)}{\Gamma(s+1+u-j)}
=\sum_{u=n}^{\infty}\frac{\Gamma(u-n+1-\alpha)}{(u+s)\Gamma(u-n+1+s)}\sum_{j=0}^{n-1}k^{\alpha+1}(j)k^s(j+u-n+1).
$$

Hence, using \cite[Eq.(1)]{ET} and that $k^s$ is decreasing for $0<s<1$, one has 
\begin{eqnarray*}
\Vert \mathfrak{g}^{0}_n\Vert_{A^{\alpha}(\D)}
&=&2-2\frac{\Gamma(1+\alpha+s)}{\Gamma(1+\alpha)\Gamma(-\alpha)}\sum_{j=0}^{n-1}k^{\alpha+1}(j)
\sum_{v=0}^{\infty}\frac{\Gamma(v+1-\alpha)}{(v+n+s)\Gamma(v+1+s)}k^s(j+v+1)\\
&\le&
2+C_{\alpha,s}\sum_{j=0}^{n-1}
\sum_{v=0}^{\infty}\frac{k^{\alpha+1}(j) k^s(j+1)}{(v+n+s)(v+s)^{s+\alpha}}
\le2+C'_{\alpha,s}\sum_{j=0}^{n-1}
\sum_{v=0}^{\infty}\frac{k^{\alpha+s}(j)}{(v+n+s)(v+s)^{s+\alpha}}\\
&=&2+C'_{\alpha,s}k^{\alpha+s+1}(n-1)\sum_{v=0}^{\infty}\frac{1}{(v+n+s)(v+s)^{s+\alpha}}\\
&=&2+C'_{\alpha,s}k^{\alpha+s+1}(n-1)\left(\frac{1}{(n+s)s^{s+\alpha}}+\int_n^{\infty}\frac{du}{u(u-n)^{s+\alpha}}\right)\\
&\leq&2+\frac{C''_{\alpha,s}}{n^{s+\alpha}}\frac{k^{\alpha+s+1}(n-1)}{n^{s+\alpha}}\leq M
\end{eqnarray*}
for some constants $C_{\alpha,s}>0$, $C'_{\alpha,s}>0$, $M>0$.

(ii) Let $n\in\N$ and $\vert z\vert<1$. Recall the notation $\mathfrak{g}_n(z)
=  \frac{1}{\mathfrak{f}(z)}
\sum_{l=0}^{n-1}D^\alpha f(l)\Delta^{-\alpha}\mathcal{Z}(l)$ in Section \ref{admisi} and put
$$
\mathfrak{s}_n(z)=\frac{(1-z)}{\mathfrak{f}(z)}\sum_{l=0}^{n-1}\Delta^{-\alpha}\mathcal{Z}(l)\sum_{j=n}^{\infty}k^{-\alpha}(j-l) f(j).
$$
Then we can write $(1-z)\mathfrak{g}^0_n(z)=(1-z)\mathfrak{g}_n(z)-\mathfrak{s}_n(z)$. Namely,
\begin{eqnarray*}
(1-z)\mathfrak{g}^0_n(z)&=&
\frac{(1-z)}{\mathfrak{f}(z)}\sum_{j=0}^{n-1} f(j)(k^{-\alpha}*\Delta^{-\alpha}\mathcal{Z})(j)
=\frac{(1-z)}{\mathfrak{f}(z)}\sum_{l=0}^{n-1}\Delta^{-\alpha}\mathcal{Z}(l)\sum_{j=l}^{n-1}k^{-\alpha}(j-l) f(j)\\
&=&(1-z)\mathfrak{g}_n(z)
-\frac{(1-z)}{\mathfrak{f}(z)}\sum_{l=0}^{n-1}\Delta^{-\alpha}\mathcal{Z}(l)\sum_{j=n}^{\infty}k^{-\alpha}(j-l) f(j)
=(1-z)\mathfrak{g}_n(z)-\mathfrak{s}_n(z).
\end{eqnarray*}
and one has $\lim_{n\to\infty}(1-z)\mathfrak{g}^0_n=1-z$ in ${A^{\alpha}(\D)}$ if and only if
$\lim_{n\to\infty}\Vert\mathfrak{s}_n\Vert_{A^{\alpha}(\D)}=0$, by Theorem \ref{BAI} (iii). To show that the second preceding limit is zero we proceed as follows.

Since $z\Delta^{-\alpha}\mathcal{Z}(l)=\Delta^{-\alpha}\mathcal{Z}(l+1)-k^{\alpha}(l+1)$
and $Dk^{-\alpha}(j)=-k^{-\alpha-1}(j+1)$ one gets
\begin{eqnarray*}
\mathfrak{f}(z)\mathfrak{s}_n(z)&=&\sum_{l=0}^{n-1}\Delta^{-\alpha}\mathcal{Z}(l)
\sum_{j=n}^{\infty}k^{-\alpha}(j-l)f(j)-\sum_{l=1}^{n}\Delta^{-\alpha}\mathcal{Z}(l)
\sum_{j=n}^{\infty}k^{-\alpha}(j-l+1)f(j)\\
&&+\sum_{l=0}^{n-1}k^{\alpha}(l+1)\sum_{j=n}^{\infty}k^{-\alpha}(j-l)f(j)\\
&=&\biggl(\sum_{j=n}^{\infty}k^{-\alpha}(j)f(j)+
\sum_{l=1}^{n}k^{\alpha}(l)\sum_{j=n}^{\infty}k^{-\alpha}(j-l+1)f(j)\biggr)\\
&&-\sum_{l=1}^{n}\Delta^{-\alpha}\mathcal{Z}(l)
\sum_{j=n}^{\infty}k^{-\alpha-1}(j-l+1)f(j)-\Delta^{-\alpha}\mathcal{Z}(n)\sum_{j=n}^{\infty}k^{-\alpha}(j-n+1)f(j)\\
&=&\mathfrak{s}^1_n(z)+\mathfrak{s}^2_n(z)+\mathfrak{s}^3_n(z),
\end{eqnarray*}
where the meaning of $\mathfrak{s}^1_n(z),\mathfrak{s}^2_n(z),\mathfrak{s}^3_n(z)$ is clear.
By (\ref{int_rep}), \eqref{deltaestimate2} and Lemma \ref{Unicity} we have
\begin{eqnarray*}
\|\mathfrak{s}^1_n\|_{A^{\alpha}(\D)}&=& \vert\sum_{j=n}^{\infty}k^{-\alpha}(j)f(j)+\sum_{l=1}^{n}k^{\alpha}(l)\sum_{j=n}^{\infty}k^{-\alpha}(j-l+1)f(j)\vert\\
&\leq & |\sum_{j=n}^{\infty}k^{-\alpha}(j)f(j)|+|\sum_{l=1}^{n}k^{\alpha}(l)\sum_{j=n}^{\infty}k^{-\alpha}(j-l+1)f(j)|\\
&\leq&-f(n)W^{-1}k^{-\alpha}(n)-f(n)\sum_{l=1}^{n}k^{\alpha}(l)W^{-1}k^{-\alpha}(n-l+1)\\
&=&k^{s}(n)k^{1-\alpha}(n-1)+k^s(n)\sum_{l=1}^{n}k^{\alpha}(l)k^{1-\alpha}(n-l)\\
&=&k^s(n)k^{1-\alpha}(n-1)+k^s(n)(k^1(n)-k^{1-\alpha}(n))\to 0,\quad \hbox{ as } n\to\infty.
\end{eqnarray*}
since, within the above sums, $f(j)>0$, $k^{\alpha}(l)>0$ and $k^{-\alpha}(j),k^{-\alpha}(j-l+1)<0$ for 
$j\geq n$ and $1\leq l\leq n,$ and $f(j)$ is decreasing. 

By a similar argument,
\begin{eqnarray*}
\|\mathfrak{s}^2_n\|_{A^{\alpha}(\D)}
&\leq&f(n)\sum_{l=1}^{n}k^{\alpha+1 }(l)\sum_{j=n}^{\infty}k^{-\alpha-1}(j-l+1)\\
&=&f(n)\sum_{l=1}^{n}k^{\alpha+1 }(l)W^{-1}k^{-\alpha-1}(n-l+1)
=-f(n)\sum_{l=1}^{n}k^{\alpha+1 }(l)k^{-\alpha}(n-l)\\
&=&-k^s(n)(k^1(n)-k^{\alpha+1}(n)-k^{-\alpha}(n))\to 0,\quad \hbox{ as } n\to\infty.
\end{eqnarray*}

Finally, using \eqref{2.10} in particular,

\begin{eqnarray*}
\|\mathfrak{s}^3_n\|_{A^{\alpha}(\D)}&= &k^{\alpha+1 }(n)|\sum_{j=n}^{\infty}k^{-\alpha}(j-n+1)f(j)|\\
&=&k^{\alpha+1 }(n)|D^{\alpha}f(n-1)-f(n-1)|\to 0,\quad  \hbox{ as } n\to\infty.
\end{eqnarray*}
\end{proof}

\begin{remark}\label{NoBAI}
\normalfont
Let $0<s<1.$ If we suppose that $\alpha\geq 1-s,$ then the sequence of functions $(\mathfrak{g}^{0}_{n})_{n\ge1}$ is not a bounded approximate identity for $\mathfrak{f}(z)=(1-z)^{-s}.$ Indeed, doing like in the proof of Theorem \ref{approximation01} (i), we get
$\displaystyle\Vert \mathfrak{g}^{0}_n\Vert_{A^{1-s}(\D)}=2-
\sum_{j=0}^{n-1}\frac{2k^{2-s}(j)}{\Gamma(2-s)\Gamma(s-1)}
\sum_{v=0}^{\infty}\frac{k^s(j+v+1)}{(v+n+s)(v+s)}$.
Since $k^s$ is decreasing and  $\sum_{j=0}^{n-1}k^{2-s}(j)=k^{3-s}(n-1)\geq c_{s}n^{2-s}$, with $c_{s}>0$, 
one has 
by \cite[Eq.(1)]{ET}
$\Vert \mathfrak{g}^{0}_n\Vert_{A^{1-s}(\D)}\geq 2+C_s \sum_{v=0}^{\infty}H_{n}(v)$,
where $\displaystyle H_{n}(v)=\frac{n^{2-s}}{(v+n+s)^{2-s}(v+s)}\to(v+s)^{-1}$ as $n\to\infty$ in an increasing way. Therefore $\lVert \mathfrak{g}^{0}_n\Vert_{A^{1-s}(\D)}$ is not uniformly bounded in $n$ and so, by \eqref{inclusion},
$\lVert\mathfrak{g}^{0}_n\Vert_{A^{\alpha}(\D)}$ is not uniformly bounded in $n$ for every $\alpha\geq 1-s$.
\end{remark}

The second theorem in this section gives an analog of Theorem \ref{approximation01} for the function $\mathfrak{L}$.

\begin{theorem}\label{approximation02}
Let $\mathfrak{f}=\mathfrak L$ where $\mathfrak L(z)=-z^{-1}\log (1-z)$
and let $\mathfrak{g}^L_{n}$ denote the function $\mathfrak{g}^0_{n}$ corresponding to $\mathfrak L$ as above. Then
\begin{itemize}
\item[(i)]  $\Vert \mathfrak{g}^{L}_n\Vert_{A^{\alpha}(\D)}\le M$ for every $n$ and $0\leq \alpha\leq 1$, with $M>0$ independent of $n$.
\item[(ii)] $\lim_{n\to\infty}(1-z)\mathfrak{g}^L_{n}(z)=1-z\ \text{in }A^{\alpha}(\D)$ for $0\leq \alpha < 1$.
\end{itemize}
\end{theorem}

\begin{proof}
(i) We have $Dg_{n}^L(0)=1$,
$Dg_{n}^L(j)=0$ for $1\leq j\leq n-2$, and $Dg_{n}^L(n-1)=-g_{n}^L(n)=-(n+1)^{-1}$ by \eqref{eq6.1}. On the other hand,
for $j\geq n$ we have $g_{n}^L(j)=\sum_{l=0}^{n-1}f(l)g(j-l)<0$ and
the $1$-admisibility of $\mathfrak{f}$ implies
$Dg_{n}^L(j)=g_{n}^L(j)-g_{n}^L(j+1)=\sum_{l=0}^{n-1}f(j)Dg(j-l)<0$.

Since $\mathfrak{g}^{L}_n\in A^{1}(\D)$, by \eqref{int_rep} one gets
$\mathfrak{g}^{L}_n(1)=\displaystyle\sum_{j=0}^{\infty}Dg_{n}^L(j)k^{2}(j)
=\frac{1}{\mathfrak{f}(1)}\sum_{j=0}^{n-1}\frac{1}{j+1}=0$ and then 
\begin{eqnarray*}
\Vert \mathfrak{g}^{L}_n\Vert_{A^{1}(\D)}&=&\displaystyle\sum_{j=0}^{\infty}|Dg_{n}^L(j)|k^{2}(j)
=\displaystyle\sum_{j=0}^{n-1}Dg_{n}^L(j)k^{2}(j)-\sum_{j=n}^{\infty}Dg_{n}^L(j)k^{2}(j)\\
&=&2\displaystyle\sum_{j=0}^{n-1}Dg_{n}^L(j)k^{2}(j)-\mathfrak{g}^{L}_n(1)=2+2\frac{n}{n+1}\leq M,
\end{eqnarray*}
with $M>0$ independent of $n$. Finally,
the inclusions given in \eqref{inclusion} imply the result.

(ii) Write $(1-z)\mathfrak{g}^L_n(z)=(1-z)\mathfrak{g}_n(z)-\mathfrak{s}_n(z)$
where $\mathfrak{g}_n$ and
$\mathfrak{s}_n=\displaystyle\frac{1}{\mathfrak{f}(z)}(\mathfrak{s}_n^1+\mathfrak{s}_n^2+\mathfrak{s}_n^3)$
have the same meaning as in the proof of Theorem \ref{approximation01}. Then, as in that theorem, it is enough to prove that
$\|\mathfrak{s}_n\|_{A^{\alpha}(\D)}\to 0$ as $n\to\infty$ to arrive at the conclusion.
To show this, note that
\begin{eqnarray*}
\|\mathfrak{s}^1_n\|_{A^{\alpha}(\D)}&\leq&
-f(n)W^{-1}k^{-\alpha}(n)-f(n)\sum_{l=1}^{n}k^{\alpha}(l)W^{-1}k^{-\alpha}(n-l+1)\\
&=&\frac{k^{1-\alpha}(n-1)}{n+1}+\frac{1}{n+1}\sum_{l=1}^{n}k^{\alpha}(l)k^{1-\alpha}(n-l)\\
&=&\frac{k^{1-\alpha}(n-1)}{n+1}+\frac{1}{n+1}(k^1(n)-k^{1-\alpha}(n))\to 0, \hbox{ as } n\to\infty,
\end{eqnarray*}

\begin{eqnarray*}
\|\mathfrak{s}^2_n\|_{A^{\alpha}(\D)}&\leq&-f(n)\sum_{l=1}^{n}k^{\alpha+1 }(l)k^{-\alpha}(n-l)\\
&=&-\frac{1}{n+1}(k^1(n)-k^{\alpha+1}(n)-k^{-\alpha}(n))\to 0, \hbox{ as } n\to\infty,
\end{eqnarray*}

and by Example \ref{ex22} (iii),

\begin{eqnarray*}
\|\mathfrak{s}^3_n\|_{A^{\alpha}(\D)}&= &k^{\alpha+1 }(n)|\sum_{j=n}^{\infty}k^{-\alpha}(j-n+1)f(j)|\\
&=&\frac{1}{n+1}|D^{\alpha}f(n-1)-f(n-1)|\to 0, \hbox{ as }  n\to\infty.
\end{eqnarray*}
\end{proof}

\begin{remark}\label{NOaiLOG}
\normalfont
Let $\alpha=1$ and write $(1-z)\mathfrak{g}^{L}_{n}(z)=1-z+g_n^L(n)-\sum_{j=n+1}^{\infty}Dg_n^L(j-1)z^j$ 
so that
$\displaystyle\|(1-z)\mathfrak{g}^{L}_{n}(z)-(1-z)\|_{A^{1}(\D)}\geq |g_n^L(n)|k^2(n-1)=\frac{n}{n+1}$. This implies that the sequence $(\mathfrak{g}^{L}_{n})_{n\ge1}$ is not an approximate identity for
$\mathfrak{L}(z)=-\displaystyle\frac{\log (1-z)}{z}$.
\end{remark}

\section{Domains of operatorial functions}\label{DOF}
\setcounter{theorem}{0}
\setcounter{equation}{0}

In this section, we characterize the domain of operators $\mathfrak{f}(T)$, given by the functional calculus associated 
with a $(C,\alpha)$-bounded operator $T$ and $\alpha$-admissible functions $\mathfrak{f}$ on $\D$, by 
transferring to operators the results of Section \ref{admisi}. The results here generalize to fractional $\alpha$ those of \cite{HT}. Part of proofs mimic those of \cite{HT}, but even in these cases we include them, for convenience of readers.

\begin{proposition}\label{prop7.1}  Let $\mathfrak{f}$ be an $\alpha$-admissible function, and $(\mathfrak{g}_n)_{n\geq 1}$ be defined as in Section 6. Let $T$ be a
$(C,\alpha)$-bounded operator on $X$ with $\hbox{Ker}(I-T)=\{0\}$.
\begin{itemize}
\item[(i)] If $\mathfrak{f}(1)<\infty$ then
$\lim_{n\to\infty}\mathfrak{g}_{n}(T)=I$ in the operator norm.
\item[(ii)] If $\mathfrak{f}(1)=\infty$ then
$\Vert \mathfrak{g}_n(T)\Vert\le 2K_{\alpha}(T)$, $n\geq 1$.
\item[(iii)]   If $\mathfrak{f}(1)=\infty,$ $(1-z)\mathfrak{f}(z)\in A^{\alpha}(\D)$ and $D^{\alpha}f(j)j^{\alpha}\to 0$ as
$j\to\infty,$ then
$$
\lim_{n\to\infty}(I-T)\mathfrak{g}_{n}(T)=I-T\quad \text{in the operator norm}.
$$
\item[(iv)] If $\mathfrak{f}(1)=\infty,$ $(1-z)\mathfrak{f}(z)\in A^{\alpha}(\D)$ and $x\in X$ is such that
$$
D^{\alpha}f(j-1)\Delta^{-\alpha}\mathcal{T}(j)\mathfrak{f}(T)^{-1}x\to 0 \hbox{ as } j\to\infty,
$$
then
$$
\lim_{n\to\infty}(I-T)\mathfrak{g}_{n}(T)x=(I-T)x\quad \text{in norm}.
$$
\end{itemize}
\end{proposition}

\begin{proof} Assertions (i), (ii) and (iii) are straightforward consequences of Theorem \ref{BAI}
and the estimate \eqref{ine} involving the $\Phi_\alpha$-functional calculus for $T$.
As regards part (iv), applying the functional calculus $\Phi_\alpha$ to \eqref{4.2} and \eqref{hache} we have
\begin{equation}\label{54}
(I-T)\mathfrak{g}_{n}(T)
=\mathfrak{h}_n(T)-D^{\alpha}f(n-1)\Delta^{-\alpha}\mathcal{T}(n)\mathfrak{f}(T)^{-1},
\end{equation}
and $\lim_{n\to\infty}\mathfrak{h}_n(T)=(I-T)$ in the operator norm. Then the result follows from the hypothesis on the second term in \eqref{54}.
\end{proof}

\begin{corollary}\label{Cor8.5} Let $\mathfrak{f}$ be an $\alpha$-admissible function such that
$(1-z)\mathfrak{f}(z)\in A^{\alpha}(\D)$, and let $T$ be a $(C,\alpha)$-bounded operator on $X$ with
$\rm{Ker}(I-T)=\{0\}.$ Then
$$
\lim_{n\to \infty}\mathfrak{g}_{n}(T)x=x, \quad x\in \overline{\rm{ Ran }}(I-T),
$$
if and only if
$$\lim_{n\to\infty}D^{\alpha}f(n-1)\Delta^{-\alpha}\mathcal{T}(n)w=0, \quad w\in \overline{\rm{ Ran }}(I-T).$$
\end{corollary}

\begin{proof}
First we prove the part ``if''.
Let $y\in X$.
Take $s\in (0,1)$.
By Proposition \ref{prop4.8} we have  $x=(I-T)^s y\in \overline{(I-T)X}$ and then we get
$\displaystyle\lim_{n\to\infty}D^{\alpha} f(n-1)\Delta^{-\alpha}\mathcal{T}(n)\mathfrak{f}(T)^{-1}x=0$ from the hypothesis.
Therefore it follows by Proposition \ref{prop7.1} (iv) that
\begin{equation}\label{71}
\lim_{n\to\infty}\mathfrak{g}_{n}(T)(I-T)^{1+s}y=(I-T)^{1+s}y\quad \text{in  norm}.
\end{equation}
Then the uniform boundedness of $\{ \mathfrak{g}_n(T)\}_{n\in\N}$ given in
Proposition \ref{prop7.1} (ii),
together with \eqref{71} and \eqref{-sto-1} imply
\begin{eqnarray*}
\| \mathfrak{g}_{n}(T)(I-T)^{1}y- (I-T)^{1}y\| &\leq &
\| \mathfrak{g}_{n}(T)\left((I-T)^{1}y- (I-T)^{1+s}y\right)\| \\
&+&\| \mathfrak{g}_{n}(T)(I-T)^{1+s}y- (I-T)^{1+s}y\|\\
&+&\| (I-T)^{1+s}y- (I-T)^{1}y\|\to 0,\quad \text{as }n\to\infty.
\end{eqnarray*}
Hence, we conclude that
$\lim_{n\to \infty}\mathfrak{g}_{n}(T)x=x$, $x\in \overline{\rm{ Ran }}(I-T)$.

Conversely, assume  $\lim_{n\to \infty}\mathfrak{g}_{n}(T)x=x$ for all
$x\in \overline{\rm{ Ran }}(I-T)$.
Then $\lim_{n\to\infty}(I-T)\mathfrak{g}_{n}(T)=(I-T)$ in the operator norm, and by \eqref{54}
one gets
$\lim_{n\to\infty}D^{\alpha}f(n-1)\Delta^{-\alpha}\mathcal{T}(n)\mathfrak{f}(T)^{-1}=0$ strongly.
Since $\mathfrak{f}(T)(I-T)\in\mathcal{B}(X)$ we have ${\rm{ Ran }}(I-T)\subset {\rm{ Dom }}\,\mathfrak{f}(T),$ 
from which one obtains $\lim_{n\to\infty}D^{\alpha}f(n-1)\Delta^{-\alpha}\mathcal{T}(n)w=0,$ for all $w\in \overline{\rm{ Ran }}(I-T)$.
\end{proof}

In the above proposition and corollary, the assumption ${\rm{Ker}}(I-T)=\{0\}$ is of a technical character since we deal with admissible functions which are regularizable through the mapping $z\mapsto 1-z$. 
As usual on results involving ergodicity, a stronger condition is assumed in the sequel. The existence of the projection $P_\beta$ holds true only on the subspace $X_\beta$, (Lemma \ref{dondeconverge}). 
Since ${\rm{Ker}}(I-T)$ is the set of fixed points of $T$ the only significant subspace of $X$ for ergodicity is
$\overline{\rm{ Ran}}(I-T)$, see for instance \cite{HT} in these respects. So we will assume $X=\overline{\rm{ Ran}}(I-T)$. Note that in this case if $T$ is $(C,\alpha)$-bounded then it is $(C,\beta)$-ergodic for all $\beta>\alpha$ (Theorem \ref{MeanErg}).

\begin{theorem}\label{dominios}
Let $\alpha>0$ and let $\mathfrak{f}$ be an $\alpha$-admissible function such that $(1-z)\mathfrak{f}(z)\in A^{\alpha}(\D)$ and $D^{\alpha}f(j)j^{\alpha}\to 0$ as $j\to\infty.$ If $T$ is a $(C,\alpha)$-bounded operator on $X$ with
$\overline{\rm{ Ran }}(I-T)=X,$ the following assertions are equivalent for a given $x$ in $X$:
\begin{itemize}
\item[(i)] $x\in {\rm{ Dom }}\,\mathfrak{f}(T)$.
\item[(ii)] The series $\sum_{j\geq 0}D^{\alpha}f(j)\Delta^{-\alpha}\mathcal{T}(j)x$ converges in norm.
\item[(iii)] The series $\sum_{j\geq 0}D^{\alpha}f(j)\Delta^{-\alpha}\mathcal{T}(j)x$ converges weakly.
\end{itemize}
Furthermore, if \rm{(i)-(iii)} holds true then
$$
\mathfrak{f}(T)x=\sum_{j\geq 0}D^{\alpha}f(j)\Delta^{-\alpha}\mathcal{T}(j)x.
$$
\end{theorem}

\begin{proof}
Since $\overline{\rm{ Ran }}(I-T)=X$ we have that $\mathfrak{g}_{n}(T)$ converges to the identity operator in the strong topology, by Corollary \ref{Cor8.5}. From the hypothesis, $\mathfrak{f}$ is $\alpha$-regularizable by $1-z$ so that there exists $\mathfrak{f}(T)$ (as closed operator).

(i)$\Rightarrow$(ii) Let $x\in \rm{ Dom }\,\mathfrak{f}(T).$ Since $\mathfrak{g}_{n}(T)\mathfrak{f}(T)\subset (\mathfrak{g}_{n}\mathfrak{f})(T),$ one gets
$$
\sum_{j=0}^{n-1}D^{\alpha}f(j)\Delta^{-\alpha}\mathcal{T}(j)x
=(\mathfrak{g}_{n}\mathfrak{f})(T)x
=\mathfrak{g}_{n}(T)\mathfrak{f}(T)x\to\mathfrak{f}(T)x,\quad n\to\infty.
$$

(ii)$\Rightarrow$(iii) This is obvious.

(iii)$\Rightarrow$(i) Assume
$\displaystyle y:=\lim_{n\to\infty}\sum_{j=0}^{n-1}D^{\alpha}f(j)\Delta^{-\alpha}\mathcal{T}(j)x$ weakly, for some $x\in X$. Let $\mathfrak{f}_n$ be as prior to Theorem \ref{lemma6.1}. Then,
$$
[(1-z)\mathfrak{f}](T)x=\lim_{n\to\infty}(I-T)\mathfrak{f}_{n}(T)x=(I-T)\sum_{j\geq 0}D^{\alpha}f(j)\Delta^{-\alpha}\mathcal{T}(j)x=(I-T)y,
$$
where the last equality is weakly. Hence $x\in\rm{ Dom }\,\mathfrak{f}(T)$ and
$$
\mathfrak{f}(T)x=(I-T)^{-1}[(1-z)\mathfrak{f}](T)x=y.
$$
\end{proof}

\begin{remark}
\normalfont 
Note that from \eqref{4.2} we have
$(I-T)\mathfrak{f}_n(T)
=(\mathfrak{f}\mathfrak{h}_n)(T)-D^{\alpha}f(n-1)\Delta^{-\alpha}\mathcal{T}(n)$.
In addition,
$D^{\alpha}f(n-1)\Delta^{-\alpha}\mathcal{T}(n)
=TD^{\alpha}f(n-1)\Delta^{-\alpha}\mathcal{T}(n-1)+D^{\alpha}f(n-1)k^{\alpha}(n)$.
Then, following ideas of \cite[Th. 5.6]{HT}, the previous identities imply that the assertions (i)-(iii) of Theorem \ref{dominios} are equivalent to the Ces\`aro convergence and/or the Ces\`aro weak convergence of  $\sum_{j\geq 0}D^{\alpha}f(j)\Delta^{-\alpha}\mathcal{T}(j)x.$
Also, similarly to \cite[Th. 5.6]{HT}, if $X$ is a reflexive Banach space then assertions (i)-(iii) of Theorem \ref{dominios} are equivalent to
$\sup_{N}\lVert \sum_{j=0}^N D^{\alpha}f(j)\Delta^{-\alpha}\mathcal{T}(j)x\rVert <\infty$.

The assertions of this remark are left to the readers for verification.
\end{remark}

\section{Fractional Poisson equation and the logarithm}\label{PoiLo}
\setcounter{theorem}{0}
\setcounter{equation}{0}

As before, we assume $T$ is a $(C,\alpha)$-bounded operator on $X$ with $(I-T)X$ dense in $X$. Then it is readily seen by \eqref{-sto-1} that $(I-T)^{s}$ --defined by the functional calculus-- is a $C_0$-semigroup of bounded operators on $X$, holomorphic in $\Real s>0$. Let $\log(I-T)$ denote the infinitesimal generator of $\left((I-T)^{s}\right)_{\Real s>0}$. Next, we discuss the solvability of the fractional Poisson equation for $T$, as well as the domain of the generator $\log(I-T)$. To do this, we apply the results on domains of operatorial functions of Section \ref{DOF}.

\medskip
{\bf Fractional Poisson equation.} This is $(I-T)^s u=x$ where $x$ is given and $u$ is the unknown. By hypothesis, $(I-T)$ is injective and so $(I-T)^{-s}:=((I-T)^{s})^{-1}$ is such that ${\rm{ Ran }}(I-T)^s={\rm{ Dom }} (I-T)^{-s}$. In fact,
$(I-T)^{-s}={\mathfrak k}^s(T)$ where ${\mathfrak k}^s=(1-z)^{-s}$ is $\alpha$-admissible; see Section \ref{ejemplos}.
Obviously, the equation is solved if and only if $x$ lies in ${\rm{ Ran }}(I-T)^s$ and the solution $u$ is
$u=(I-T)^{-s}x$. First, we characterize the occurrence of $x$ in ${\rm{ Ran }}(I-T)^s$ through convergence of series involving 
Ces\`aro sums of $T$.

\begin{theorem}\label{CesPoisson} Let $T$ be a $(C,\alpha)$-bounded operator on $X$ with $\overline{(I-T)X}=X$ and let 
$0<s<1$.
For $x\in X$ the following assertions are equivalent:
\begin{itemize}
\item[(i)] $x\in {\rm{ Ran }}(I-T)^s$.
\item[(ii)] The series $\displaystyle\sum_{n=1}^\infty\frac{1}{n^{1+\alpha-s}}\Delta^{-\alpha}\mathcal{T}(n)x$ converges in norm.
\item[(iii)] The series $\displaystyle\sum_{n=1}^\infty\frac{1}{n^{1+\alpha-s}}\Delta^{-\alpha}\mathcal{T}(n)x$ converges
weakly in $X$.
\end{itemize}

Furthermore, if (i), (ii) or (iii) hold true then
$$
(I-T)^{-s}x=\frac{\sin(\pi s)\Gamma(1-s+\alpha)}{\pi}
\sum_{n=0}^\infty\frac{\Gamma(s+n)}{\Gamma(n+\alpha+1)}\Delta^{-\alpha}\mathcal{T}(n)x.
$$
\end{theorem}

\begin{proof}
The equivalence between (i) and the either weak or in-norm convergence of the series
$$
\frac{\sin(\pi s)\Gamma(1-s+\alpha)}{\pi}
\sum_{n=0}^\infty\frac{\Gamma(s+n)}{\Gamma(n+\alpha+1)}\Delta^{-\alpha}\mathcal{T}(n)x
$$
follows by Theorem \ref{dominios} and \eqref{2.10}. To finish the proof is then enough to apply \eqref{2.11}. 
\end{proof}

The mean ergodic theorem for power-bunded operators $T$
says that
$\sup_n\Vert M_T^0(n)\Vert<\infty$ jointly with ${\overline{\rm{ Ran }}}(I-T)=X$ imply that $M_T^1(n)x$ converges (to $0$), 
as $n\to\infty$, for every $x\in X$. Further, if $x\in{\rm{ Ran}}(I-T)^s$ with $0<s<1$ then the  rate of convergence 
is
$\Vert M_T^1(n)x\Vert=o(n^{-s})$ as $n\to\infty$, \cite{DL, HT}. We have something similar for $(C,\alpha)$-bounded operators.

\begin{corollary}\label{GenRatePoisson}
Let $\alpha>0$ and let $T$ be a $(C,\alpha)$-bounded operator on $X$ where $X={\overline{\rm{ Ran }}}(I-T)$. Then
$\lim_{n\to\infty} M_T^{\alpha+1}(n)x=0$ for every $x\in X$. Moreover, if $x\in{\rm{ Ran}}(I-T)^s$ with $0<s<1$ then
$$
\Vert M_T^{\alpha+1}(n)x\Vert=o (n^{-s}), \quad \hbox{ as } n\to\infty.
$$
\end{corollary}

\begin{proof}
The first part of the corollary is a direct consequence of Theorem \ref{MeanErg}. As for the rate of convergence note that
by Theorem \ref{CesPoisson} the series
$\sum_{n=1}^\infty n^{s-1-\alpha}\Delta^{-\alpha}\mathcal{T}(n)x$ is convergent in norm. Then, by Kronecker's Lemma \cite[p. 103]{DL},
$\lim_{n\to\infty}\Vert n^{s-\alpha-1}\sum_{j=1}^n\Delta^{-\alpha}\mathcal{T}(j)x\Vert=0$
and therefore
$$
\Vert M_T^{\alpha+1}(n)x\Vert=\left\Vert {(k^1\ast k^\alpha\ast\mathcal T)(n)\over k^{\alpha+2}(n)}x\right\Vert
\sim\Vert{1\over n^{\alpha+1}}\sum_{j=1}^n\Delta^{-\alpha}\mathcal{T}(j)x\Vert=o(n^{-s}),
$$
as $n\to\infty$, as we wanted to show.
\end{proof}

QUESTION: Since $T$ being $(C,\alpha)$-bounded on $X={\overline{\rm{ Ran }}}(I-T)$ implies $T$ is $(C,\beta)$-ergodic
for every $\beta>\alpha$, is it possible to obtain, similarly to the above Corollary \ref{GenRatePoisson}, 
the rate of convergence 
(to $0$) of
$\Vert M_T^{\alpha+\varepsilon}(n)x\Vert$ for every $\varepsilon>0$ ?

\medskip
The fact that the rate of convergence of the series in the above theorem and corollary are given in terms of (fractional) 
Ces\`aro sums seems to be natural, on account of the general character of $(C,\alpha)$-bounded operators.
However, if one restricts the range of fractional $\alpha$ to values between $0$ and $1-s$, it is then possible to express convergence involving only the Taylor series of $(I-T)^{-s}$. The next theorem extends the corresponding results obtained for power-bounded operators in \cite{DL,HT}.

\begin{theorem}\label{equivTaylor}
Let $s$ be such that $0<s<1$. Take $\alpha$ such that $0\le\alpha<1-s$. Let $T$ be a $(C,\alpha)$-bounded operator on $X$ with $\overline{(I-T)X}=X$. Then the following assertions are equivalent.
\begin{itemize}
\item[(i)] $x\in {\rm{ Ran }}(I-T)^s$.
\item[(ii)] The series $\displaystyle\sum_{n=1}^\infty\frac{1}{n^{1-s}}T^n x$ converges in norm (or weakly).
\end{itemize}

In any of the aboves cases,
$$
(I-T)^{-s}x=\sum_{n=0}^\infty k^s(n)T^nx.
$$
\end{theorem}

\begin{proof}
The argument follows similar lines to the ones of previous Theorem \ref{CesPoisson} and Theorem \ref{dominios}, 
by using the approximate unit ${\mathfrak g}_n^0$ of Section \ref{taylor} instead ${\mathfrak g}_n$. In some detail, 
suppose first that $x$ belongs to
${\rm{ Ran }}(I-T)^s$. Then
${\mathfrak g}_n^0(T){\mathfrak k}^s(T)\subset ({\mathfrak g}_n^0{\mathfrak k}^s)(T)$ and hence we have
\begin{eqnarray*}
\sum_{n=0}^{ N-1} k^s(n)T^nx&=&\left(\sum_{n=0}^{ N-1} k^s(n)T^n\right){\mathfrak k}^{-s}(T){\mathfrak k}^s(T)x\\
&=&{\mathfrak g}_N^0(T)(I-T)^{-s}x\longrightarrow(I-T)^{-s}x,\quad N\to\infty.
\end{eqnarray*}
since $\lim_{N\to\infty}{\mathfrak g}_N^0= I$ strongly on $X$ by Theorem \ref{approximation01}. Conversely,
suppose now that there exists $y:=\lim_{N\to\infty}\sum_{n=0}^{N-1} k^s(n)T^nx$ for some $x\in X$ (weakly or in norm). Then
$$
[(1-z){\mathfrak k}^s](T)x=\lim_{\N\to\infty}(I-T){\mathfrak k}^s(T){\mathfrak g}_N^0(T)x
=
\lim_{\N\to\infty}(I-T)\left(\sum_{n=0}^{N-1} k^s(n)T^nx \right)=(I-T)y,
$$
so that $x\in {\rm{ dom }}(I-T)^{-s}$ and $(I-T)^{-s}x=y$. The equivalence between (i) and (ii) follows now from \eqref{double2}.
\end{proof}

\begin{corollary}\label{SmallRatePoisson}
Let $s$ be such that $0<s<1$ and let $\alpha$ be such that $0\le\alpha<1-s$. Assume that $T$ is a $(C,\alpha)$-bounded operator on $X$ with $X={\overline{\rm{ Ran }}}(I-T)$. Then
$\lim_{n\to\infty}n^{-1}\sum_{j=1}^nT^jx=0$ for every $x\in X$. Moreover, if $x\in{\rm{ Ran}}(I-T)^s$ then
$$
\Vert{1\over n}\sum_{j=1}^nT^jx\Vert=o(n^{-s}), \quad \hbox{ as } n\to\infty.
$$
\end{corollary}

\begin{proof}
Note that $n^{-1}\sum_{j=1}^nT^j=M_T^1(n)$ for every $n$. Then the proof of this corollary is similar to that one of
Corollary \ref{GenRatePoisson}, using in particular (vector) Kronecker's Lemma once again.
\end{proof}

\medskip
As pointed out in Remark \ref{NoBAI}, $({\mathfrak g}_n^0)_{n\ge1}$ is not a bounded approximate unit in
$A^\alpha(\D)$ when
$\alpha\ge1-s$ so that the above type of argument does not work in this case to prove the equivalence established in
Theorem \ref{equivTaylor}.

\medskip
{\bf The operatorial logarithmic function and the discrete Hilbert transform.}
For $z\in\D$, there is the decomposition
$$
\displaystyle\log(1-z)=-\sum_{n\geq 1}\frac{z^j}{n}=\mathfrak{h}(z)-\mathfrak{L}(z)
$$
where
$$
\mathfrak{h}(z):=1-\sum_{n=1}^{\infty}\frac{z^n}{n(n+1)},\quad \mathfrak{L}(z):=\sum_{n=0}^{\infty}\frac{z^n}{n+1}.
$$
The function $\mathfrak{L}$, with Taylor coefficients $L(n):=(n+1)^{-1}$, has been studied in Example \ref{Ex64},
where it has been shown that it is an $\alpha$-admissible function with
$D^{\alpha}L(n)n^{\alpha}\to 0$ as $n\to\infty$, and that $(L(n))_{n\in\N_0}$ is log-convex of any order.
Note that $(1-z)\mathfrak{L}=\mathfrak{h}\in A^{\alpha}(\D)$ for all $\alpha\geq 0$.

The functional calculus of Section \ref{CFcesar} enables us to define the closed operator 
$\log(I-T):=[\log(I-z)](T)$ by regularization. It is readily seen that $\log(I-T)$ is the infinitesimal generator
 of the holomorphic semigroup
$(I-T)^s$, $\Real s>0$.

\begin{theorem}\label{LOG}
Let $T$ be a $(C,\alpha)$-bounded operator on $X$ with $\overline{\rm{ Ran }}(I-T)=X.$
Given $x\in X$ the following are equivalent:
\begin{itemize}
\item[(i)] $x\in {\rm{ Dom }}(\log(I-T))$.
\item[(ii)] The series
$\displaystyle\sum_{n=1}^\infty\frac{1}{n^{1+\alpha}}\Delta^{-\alpha}\mathcal{T}(n)x$ converges in norm (or weakly). 
\end{itemize}
Furthermore, if (i) or (ii) holds true then
$$
\log(I-T)x=\left(\int_0^1{1-u^\alpha\over1-u}\ du\right)x
-\sum_{n=1}^\infty\frac{\Gamma(\alpha+1)\Gamma(n)}{\Gamma(n+\alpha+1)}\Delta^{-\alpha}\mathcal{T}(n)x.
$$
\end{theorem}

\begin{proof}
Let $\mathfrak{h}$ and $\mathfrak{L}$ be as prior to the theorem. Then we have
$(1-z)\log(1-z)=(1-z)\mathfrak{h}(z)-(1-z)\mathfrak{L}(z)=(1-z)\mathfrak{h}(z)-\mathfrak{h}(z)$ in $A^{\alpha}(\D)$
and therefore
$(I-T)\log(I-T)=(I-T)\mathfrak{h}(T)-\mathfrak{h}(T)$ in $\mathcal B(X)$, from which one obtains
$$
\log(I-T)=(I-T)^{-1}[(I-T)\mathfrak{h}(T)]-(I-T)^{-1}\mathfrak{h}(T)=\mathfrak{h}(T)-\mathfrak{L}(T)
$$
as closed operators on $X$. Moreover $\mathfrak{h}(T)$ is bounded and so the domains of $\log(I-T)$ and $\mathfrak{L}(T)$ coincide. Hence $x\in {\rm{ Dom }}(\log(I-T))$ if and only
$\displaystyle\sum_{n\geq 0}\frac{n!}{\Gamma(n+\alpha+2)}\Delta^{-\alpha}\mathcal{T}(n)x$ converges (in norm or weakly), according to Theorem \ref{dominios} and \eqref{2.12}. To get the equivalence between (i) and (ii) is now enough to use
\eqref{2.13}. 

As regards the range of $\log(I-T)$, note that on ${\rm{ Dom }}(\log(I-T))$ we have $\log(I-T)
=\mathfrak{h}(T)-\mathfrak{L}(T)=(1-T)\mathfrak{L}(T)-\mathfrak{L}(T)=-T\mathfrak{L}(T)$ and so, 
by Theorem \ref{dominios} and \eqref{2.12},
\begin{eqnarray*}
\log(I-T)x&=&-\sum_{n=0}^{\infty}D^{\alpha}L(n)T\Delta^{-\alpha}\mathcal{T}(n)x\\
&=&-\sum_{n=0}^{\infty}D^{\alpha}L(n)\Delta^{-\alpha}\mathcal{T}(n+1)x
+\sum_{n=0}^{\infty}D^{\alpha}L(n)k^\alpha(n+1)x\\
&=&-\sum_{n=1}^\infty\frac{\Gamma(\alpha+1)\Gamma(n)}{\Gamma(n+\alpha+1)}\Delta^{-\alpha}\mathcal{T}(n)x
+\sum_{n=0}^{\infty}{\alpha\over(n+1)(n+\alpha+1)}
\end{eqnarray*}
for every $x\in{\rm{ Dom }}(\log(I-T))$.

Finally,
\begin{eqnarray*}
\sum_{n=0}^{\infty}{\alpha\over(n+1)(n+\alpha+1)}
&=&\int_0^\infty\sum_{n=0}^{\infty}\left(e^{-(n+1)t}-e^{-(n+\alpha+1)t}\right) dt\\
&=&\int_0^\infty{1-e^{-\alpha t}\over1-e^{-t}}e^{-t}\ dt
=\int_0^1{1-u^\alpha\over1-u}\ du
\end{eqnarray*}
and we have done.
\end{proof}

\begin{remark}\label{hilbert}
\normalfont
Suppose for a moment that the operator $T$ is power-bounded on the Banach space $X$.
Then the formal expression
\begin{equation}\label{hilbertT}
H_T:=\sum_{n=1}^\infty{1\over n}T^n
\end{equation}
defines a closed operator on $X$ which is called the one-sided ergodic Hilbert transform, see \cite{DL} and \cite{HT}. In fact, $H_T=-\log(I-T)$ with ${\rm{ Dom }}(H_T)={\rm{ Dom }}(\log(I-T))$ in particular \cite{CCL, HT}.

Let us assume again that $T$ is a $(C,\alpha)$-bounded operator on $X$ for $\alpha>0$, as usually in the present paper. Then the above presentation of $H_T$ does not look very suitable for $T$ in principle. We wish to point out here that the rate of convergence occurring in (ii) of Theorem \ref{LOG} corresponds to a $(C,\alpha)$-bounded operator in the same way as the rate of convergence of $H_T$ corresponds to power-bounded operators. So one could define for $T$ its (one-side) Ces\`aro-Hilbert transform of order $\alpha$ (or $\alpha$-ergodic Hilbert transform, for short)
$H_T^{(\alpha)}$ as the operator given by
$$
H_T^{(\alpha)}x=
\sum_{n=1}^\infty\frac{\Gamma(\alpha+1)\Gamma(n)}{\Gamma(n+\alpha+1)}\Delta^{-\alpha}\mathcal{T}(n)x
+c_\alpha x,
$$
with $c_\alpha=-\int_0^1(1-u^\alpha)(1-u)^{-1}du$, obtained from \eqref{hilbertT} \lq\lq summing by parts up to order $\alpha$".

Then Theorem \ref{LOG} establishes that the infinitesimal generator $\log(I-T)$ of the semigroup $(I-T)^s$ and
$-H_T^{(\alpha)}$ coincide:
$$
\log(I-T)x=-H_T^{(\alpha)}x, \quad \forall x\in X \hbox{ such that }
\sum_{n=1}^\infty{1\over n^{1+\alpha}}\Delta^{-\alpha}\mathcal{T}(n)x
\hbox{ converges in } X.
$$
In this way, the theorem generalizes \cite[Prop. 3.3]{CCL}  and \cite[Th. 6.2]{HT}.
\end{remark}

\medskip
We now continue the discussion on the logarithm to show that the operator $-\log(I-T)$ is the genuine original Hilbert transform not only for power-bounded operators, but also for $(C,\alpha)$-bounded operators in the case
$0<\alpha<1$. This result extends \cite[Prop. 3.3]{CCL}  and \cite[Th. 6.2]{HT}.

\begin{theorem}\label{LOG&hilbert} Let $\alpha$ be such that $0<\alpha<1$. Let $T$ be a $(C,\alpha)$-bounded operator on $X$ with
$\overline{\rm{ Ran }}(I-T)=X$. For a given $x\in X$ the following are equivalent:
\begin{itemize}
\item[(i)] $x\in {\rm{ Dom }}(\log(I-T))$.
\item[(ii)] The series $\displaystyle\sum_{n=1}^\infty\frac{1}{n}T^nx$ converges (in norm or weakly).
\end{itemize}
In any of the above cases
$$
\log(I-T)x=-\sum_{n=1}^\infty\frac{1}{n}T^nx.
$$
\end{theorem}

\begin{proof}
As in the proof of Theorem \ref{LOG}, $\log(I-T)=\mathfrak{h}(T)-\mathfrak{L}(T)$ so that
$x\in{\rm{ Dom }}(\log(I-T))$ if and only if $x\in{\rm{ Dom }}(\mathfrak{L}(T))$. On the other hand one has that $x\in{\rm{ Dom }}(\mathfrak{L}(T))$ if and only if $\displaystyle\sum_{n=0}^\infty\frac{1}{n+1}T^nx$ converges, and in this case
$\displaystyle{\mathfrak{L}}(T)x=\sum_{n=0}^\infty\frac{1}{n+1}T^nx$. The latter assertion can be proved by the same argument as in the proof of Theorem \ref{equivTaylor}, by replacing there $(I-T)^s$ with $\mathfrak{L}(T)$ now, and
$\mathfrak{g}_n^0(T)$ with ${\mathfrak{g}}_n^L(T)$ 
(recall that ${\mathfrak{g}}_n^L$ is a bounded approximate identity for $(1-z)$ in
$A^\alpha(\D)$; see Theorem \ref{approximation02}). The identity
$\displaystyle\log(I-T)x=-\sum_{n=1}^\infty\frac{1}{n}T^nx$, for $x\in{\rm{ Dom }}(\log(I-T))={\rm{ Dom }}(\mathfrak{L}(T))$, follows then from the identity
$\log(I-T)=\mathfrak{h}(T)-\mathfrak{L}(T)$, taking $n$-partial sums and letting $n\to\infty$.
\end{proof}

\begin{remark}
\normalfont
As pointed out in Remark \ref{NOaiLOG}, $({\mathfrak g}_n^L)_{n\ge1}$ is not an approximate unit in
$A^\alpha(\D)$ when
$\alpha=1$ and therefore a type of argument as above is not enough to prove the equivalence established in
Theorem \ref{LOG&hilbert}, in this case. Also, as regards part (ii) of that theorem,  the decreasing rate to $0$ of 
$M^1_{T}(n)$ is a non simple matter which requires special treatment. For every power-bounded $T$, it has been elucidated in \cite{GHT}.
\end{remark}

\medskip
Related to the study of the logarithmic operator is the question that whether or not the equality
$\bigcup_{s>0}{\rm{ Ran }}(I-T)^s={\rm{ Dom }}(\log(I-T))$ holds true. In \cite[Th. 6.3]{HT}, the question was answered in the negative for every power-bounded operator $T$ with $1\notin \sigma(T)$. The argument to prove this result  relies on the sectorial functional calculus (Th. \ref{Sectorial}) and can be mimicked in the case of  $(C,\alpha)$-bounded operators. So we have the following.

\begin{proposition}
Let $\alpha>0$ and let $T$ be a $(C,\alpha)$-bounded operator on a Banach space $X$ such that
${\overline{\rm{Ran}}}(I-T)=X$ and $1\notin \sigma(T)$. Then we have the inclusion
$$
\bigcup_{s>0}{\rm{ Ran }}(I-T)^s\subset{\rm{ Dom }}(\log(I-T))
$$
and it is strict.
\end{proposition}

\section{Application to concrete operators}\label{apliques}
\setcounter{theorem}{0}
\setcounter{equation}{0}

In this section we show two examples 
to illustrate results of the paper.

\begin{example}\label{ONE}
\normalfont
Let $1\le p<\infty.$ With $\Vert T\Vert_{op,p}$ we denote the operator norm of operators $T$ in ${\mathcal B}(L^p(0,1))$. 
Let $V$ be the Volterra integral operator on $L^p(0,1)$ given by
$$
Vf(t):=\int_0^t f, \quad t\in[0,1],\, f\in L^p(0,1).
$$

Define $T_V:=I-V$. Estimates involving powers of the operator $T_V$ were given in \cite[Th. 11]{Hi} for $p=1$. Long after that, such estimates were extended to arbitrary $p\in[1,\infty)$. Namely, there exist $A>0$, $B>0$ such that
$$
A\ n^{\vert (1/4)-(1/2p)\vert}\le\Vert T_V^n\Vert_{op,p}\le B\ n^{\vert (1/4)-(1/2p)\vert}, \quad n\in\N;
$$
see \cite[Th. 2.2]{MSZ}. Thus $T_V$ is power-bounded on $L^p(0,1)$ exclusively in the Hilbertian case $p=2$.

In fact, powers $T^n_V$, $n\in\N$, and 
means $M^\alpha_{T_V}$ of $T_V$ can be expressed in the integral form
$$
M^\alpha_{T_V}(n) f(t)=f(t)-{1\over k^{\alpha+1}(n)}\int_0^t L_{n-1}^{(\alpha+1)}(t-u) f(u) du,\quad t\in[0,1],\, n\in\N,\, f\in L^p(0,1),
$$
where $\alpha\ge0$ ($\alpha=0$ gives the integral for $T^n_V$)
and $L_{n-1}^{(\alpha+1)}$ is the generalized Laguerre polynomial of degree $n-1$,
see \cite[(5.3) and  (6.14)]{Hi}.
In other words, we have
\begin{equation}\label{Mconv}
M^\alpha_{T_V}(n) f=(\delta_0-[k^{\alpha+1}(n)]^{-1}L_{n-1}^{(\alpha+1)})\star f,\quad n\in\N,\,f\in L^p(0,1),
\end{equation}
where $\delta_0$ is the Dirac mass at $\{0\}$ and \lq\lq$\star$" is the convolution operation on $(0,1)$. Moreover, the sequence
$$
\Lambda_{n-1}^{(\alpha+1)}:=[k^{\alpha+1}(n)]^{-1}L_{n-1}^{(\alpha+1)}, \quad n\in\N,
$$
is a bounded approximate unit in the convolution Banach algebra $L^1(0,1)$
\cite[Lemma 1 and (6.14)]{Hi}. Using these properties, it can be shown that
$T_V$ is $(C,\alpha)$-bounded on $L^1(0,1)$ if and only if $T_V$ is $(C,\alpha)$-ergodic on $L^1(0,1)$ 
if and only if $\alpha>1/2$ \cite[Th. 11]{Hi}. For $1<p<\infty$ we have the following result.

\begin{proposition}\label{volterra}
Let $V$ be the Volterra operator and set $T_V=I-V$ as above. Let $1\le p<\infty$. We have the following properties.
\begin{itemize}
\item[(i)] $(I-T_V)(L^p(0,1))$ is dense in $L^p(0,1)$.

\item[(ii)] For every $\alpha>\vert1/2-1/p\vert$ the operator $T_V$ is 
$(C,\alpha)$-ergodic on
$L^p(0,1)$.

\item[(iii)] Let $\alpha>\vert1/2-1/p\vert$ and $0<s<1$. Take $f\in L^p(0,1)$. Then the Volterra integral equation
\begin{equation}\label{EquVolterra}
{1\over\Gamma(s)}\int_0^t(t-u)^{s-1}g(u)\ du=f(t), \quad 0\le t\le1,
\end{equation}
has a (unique) solution $g\in L^p(0,1)$ if and only if 
$\sum_{n=1}^\infty n^{s-1}(\delta_0-\Lambda_{n-1}^{(\alpha+1)})\star f$ 
is norm-convergent. 
In this case,
$$
g={\sin(\pi s)\over\pi}{\Gamma(1-s+\alpha)\over\Gamma(\alpha+1)}
\sum_{n=0}^\infty{\Gamma(s+n)\over n!}(\delta_0-\Lambda_{n-1}^{(\alpha+1)})\star f.
$$
Also,
$$
\Vert(\delta_0- \Lambda_{n-1}^{(\alpha+2)})\star f\Vert_{p}=o(n^{-s}), \hbox{ as } n\to\infty.
$$
\medskip

If, moreover, $0<s<1- \vert1/2-1/p\vert$ and $\alpha\in(\vert1/2-1/p\vert,1-s),$ the 
equation \eqref{EquVolterra} has a (unique) solution $g\in L^p(0,1)$ if and only if  
$\sum_{n=1}^\infty n^{s-1}(\delta_0-L_{n-1}^{(1)})\star f$
is norm-convergent,
with
$$
g=\sum_{n=0}^\infty k^s(n)(\delta_0-L_{n-1}^{(1)})\star f.
$$
Also,
$$
{1\over n}\Vert\sum_{j=1}^n(\delta_0-\Lambda_{j-1}^{(1)})\star f\Vert_{p}=o(n^{-s}), \hbox{ as } n\to\infty.
$$

\item[(iv)] Let $\alpha>\vert1/2-1/p\vert$. Then 
$f\in{\rm{Dom}}(\log V)\subset L^p(0,1)$ if and only if the series  
$\sum_{n=1}^\infty n^{-1}(\delta_0-\Lambda_{n-1}^{(\alpha+1)})\star f$ is norm-convergent. 

In this case,
$$
(\log V)f=(\psi(\alpha+1)-\psi(1))f-\sum_{n=1}^\infty{1\over n}(\delta_0-\Lambda_{n-1}^{(\alpha+1)})\star f.
$$
If moreover $\alpha\in(\vert1/2-1/p\vert,1)$ then
$$
f\in{\rm{Dom}}(\log V)\subset L^p(0,1)\Longleftrightarrow\sum_{n=1}^\infty n^{-1}(\delta_0- L_{n-1}^{(1)})\star f\hbox{ is norm-convergent. }
$$
In this case,\,
$\displaystyle(\log V)f=-\sum_{n=1}^\infty{1\over n}(\delta_0-L_{n-1}^{(1)})\star f$.
\end{itemize}
\end{proposition}

\begin{proof}
(i) This is standard.

(ii) For fixed $n\in\N$, the mapping $z\mapsto k^z(n)$ is an entire function by (\ref{definition}). This simple but primordial fact allows one to show by standard arguments on complex interpolation (using for example 
\cite[Th. 1]{CJ}; details are left to readers) that $T_V$ is $(C,\alpha)$-bounded on $L^p(0,1)$ for every 
$\alpha>\vert (1/2)-(1/p)\vert$. So $T_V$ is $(C,\alpha)$-ergodic for every $\alpha>\vert (1/2)-(1/p)\vert$ by part (i) and Theorem \ref{MeanErg}.

(iii) The Volterra equation is $(I-T_V)^sg=V^sg=f$ and so the assertions of this part (iii) follow from (i), (\ref{Mconv}), Theorem \ref{CesPoisson}, Corollary \ref{GenRatePoisson}, Theorem \ref{equivTaylor} and Corollary \ref{SmallRatePoisson}, respectively.

(iv) This is as in part (iii), this time applying Theorem \ref{LOG} and Theorem \ref{LOG&hilbert}.
\end{proof}
\end{example}

\begin{remark}
\normalfont
In the above proposition, condition $\alpha>\vert (1/2)-(1/p)\vert$ in the results about $L^p([0,1])$ is a sufficient condition. It is pertinent to ask if that condition is also necesary, but this is not part of our aims here and then we do not address this (involved) problem in this paper.
Note also that for $p=\infty$, which corresponds to $\alpha>1/2$, all the results in Proposition \ref{volterra} hold true by replacing the space
$L^\infty([0,1])$ with $C([0,1])$.

As regarding convergence, the norm-convergence can be replaced with weak convergence in the statement of Proposition \ref{volterra}.
\end{remark}

\begin{example}\label{TWO}
\normalfont
Let $0<\beta<1$ and let $\ell^2_{\beta}(\N_0)$ denote the Hilbert space of sequences $f$ such that 
$\Vert f\Vert_{2,\beta}^2:=\sum_{j=0}^{\infty}\vert f(j)\vert^2 k^\beta(j)<\infty$.
Let $T_S$ be the backward shift operator on $\ell^2_{\beta}(\N_0)$ given by  
$$
(T_Sf)(j)=f(j+1), \quad f\in\ell^2_{\beta}(\N_0), j\in \N_0
$$
Then $\|T_S^n\|^2\sim (n+1)^{1-\beta}$, so $T_S$ is not power-bounded on 
$\ell^2_{\beta}(\N_0)$, but 
$T_S$ is $(C,\alpha)$-bounded for $\alpha>(1-\beta)/2$, see \cite{ABY}. 

We have that $I-T_S$ is the first order finite difference operator $W=D$ (recall Section \ref{Pre}), that is,
$$
(I-T_S)f(n)=f(n)-f(n+1), \qquad n\in\N_0.
$$

It is very simple to show that the space $c_{00}(\N_0)$ of eventually null sequences satisfies
$c_{00}(\N_0)\subset(I-T_S)(c_{00}(\N_0))$, whence one gets the density of $(I-T_S)(\ell^2_{\beta}(\N_0))$ in 
$\ell^2_{\beta}(\N_0)$.
As a consequence,
$T_S$ is $(C,\alpha)$-ergodic for $\alpha>(1-\beta)/2$ by Theorem \ref{MeanErg}. 
Thus we can apply the results of Section \ref{PoiLo}
to $T_S$ in a similar way we have done in the above example for $T_V$.

\begin{proposition}\label{differenceOp}
Let $T_S$ be the backward shift acting on $\ell^2_{\beta}(\N_0)$, $0<\beta<1$,
as above and assume $\alpha>(1-\beta)/2$.
\begin{itemize}

\item[(i)] Let $0<s<1$. Take $f\in \ell^2_{\beta}(\N_0)$. Then the elliptic problem in differences
\begin{equation}\label{Elliptic}
D^s u=f
\end{equation}
has a (unique) solution $u\in\ell^2_{\beta}(\N_0)$ if and only if 
$\sum_{n=1}^\infty n^{s-1-\alpha}\sum_{j=0}^n k^\alpha(n-j)f(j+\cdot)$ is norm-(or weak-)convergent in $\ell^2_{\beta}(\N_0)$, and in this case,

Also, $\displaystyle
\Vert \sum_{j=0}^n(n-j)^{\alpha-1} f(j+\cdot)\Vert_{2,\beta}=o(n^{\alpha+1-s}), \hbox{ as } n\to\infty$.

If, moreover, $0<s<(1+\beta)/2$ and $(1-\beta)/2<\alpha<1-s,$ then the equation (\ref{Elliptic})
has a (unique) solution $u\in\ell^2_{\beta}(\N_0)$ if and only if 
$\displaystyle\sum_{n=1}^\infty\frac{1}{n^{1-s}}f(n+\cdot)$ is convergent in $\ell^2_{\beta}(\N_0)$
and then the solution $u$ is given by
$\displaystyle u=\sum_{n=0}^\infty k^s(n)f(n+\cdot)=W^{-s}f$.
\item[(ii)] One has $f\in{\rm{Dom}}(\log D)\subset \ell^2_{k^\beta}(\N_0)$ if and only if 
$\displaystyle\sum_{n=1}^\infty{1\over n^{\alpha+1}}\sum_{j=0}^n k^\alpha(n-j) f(j+\cdot)$ is convergent in 
$\ell^2_{\beta}(\N_0)$.
In this case,
$$
(\log D)f=(\psi(\alpha+1)-\psi(1))f-\sum_{n=1}^\infty B(\alpha+1,n)\sum_{j=0}^n k^\alpha(n-j) f(j+\cdot).
$$
If, moreover, $\alpha\in((1-\beta)/2,1)$ then $f\in{\rm{Dom}}(\log D)\subset \ell^2_{k^\beta}(\N_0)$ if and only if
$\displaystyle
\sum_{n=1}^\infty{f(n+\cdot)\over n}$ is convergent in $\ell^2_{\beta}(\N_0)$
and then $\displaystyle(\log D)f=-\sum_{n=1}^\infty{1\over n}f(n+\cdot)$.
\end{itemize}
\end{proposition}

\begin{remark}
\normalfont
The space $\ell^2_{k^\beta}(\N_0)$ coincides, up to equivalent norms, with the weighted Bergman space 
${\mathcal B}_{\nu}$ for
$\nu=-\beta$ formed by the holomorphic functions $\frak f$ on the unit disc such that
$$
\Vert {\frak f}\Vert_{\nu,2}:=\left(\int\int_\D (\nu+1) | f(z) |^2(1-\vert z\vert^2)^\nu dz d\overline{z}\right)^{1/2}<\infty$$ 
(usually, $\nu$ is written like $\nu=\mu-2$ with $\mu>1$).
Naturally, the operator $T_S$, transferred on ${\mathcal B}_{\nu}$, reads 
$\displaystyle T_S{\frak f}(z)={{\frak f}(z)-{\frak f}(0)\over z}$, $\vert z\vert<1$,
whence one gets
$\displaystyle T_S^n {\frak f}(z)
={1\over z^n}({\frak f}(z)-\sum_{j=0}^{n-1}{{\frak f}^{(j)}(0)\over(n-1)!}z^{n-1})$, $\vert z\vert<1,\, n\in\N$
and so
$$
(I-T_S){\frak f}(z)={(z-1){\frak f}(z)-{\frak f}(0)\over z}, \quad \vert z\vert<1,
$$
which are quite more manageable using Taylor coefficients, that is, $\ell^2_{k^\beta}(\N_0)$.
\end{remark}
\end{example}

In Example \ref{ONE} we are interested in $L^p$ spaces whereas Example \ref{TWO} is rather concerned about weights. The results of Proposition \ref{differenceOp} admit translation to the $\ell^p$ case.

\section{Appendix}
\setcounter{theorem}{0}
\setcounter{equation}{0}

As it is announced formerly, we show here the proof of Theorem \ref{logconvex}. This needs a preparatory lemma.

\begin{lemma}\label{Lemma7.6}
Let $m\in\N.$ In the setting of Theorem \ref{logconvex}, $\frak f$ is zero-free on $\D$. 
Assume additionally that $\frak f$ is $p-1$-admissible with $p\in\N$ and $p\leq m.$ Then
for $v\ge1$,
$$
\sum_{l=0}^v W^p g(l)\sum_{j=v-l}^v k^p(j+l-v)D^p f(j)
=\sum_{l=v+1}^\infty W^p g(l)\sum_{j=v+1}^\infty k^p(j+l-v) D^p f(j).
$$
\end{lemma}
\begin{proof}
The fact that $\frak f$ has no zeroes in $\D$ is in \cite[Prop. 4.4]{HT}.

To prove the statement, it is enough to show  the equality
\begin{equation}\label{7}W^{p}(f*g)(v)
=\left(\sum_{l=0}^v\sum_{j=v-l}^v-\sum_{j=v+1}^\infty\sum_{l=v+1}^\infty\right)
k^p(l+j-v)D^p f(j)W^p g(l),\quad v\in\N,\end{equation} since $W^{p}(f*g)(v)=0$ for $v\in\N$ because 
$\displaystyle \mathfrak{f}\mathfrak{\frac{1}{\mathfrak{f}}}=1.$

Fix $\vert z\vert<1$. We have 
$\Delta^{-p}\vert {\mathcal Z}\vert(j)\sim (j+1)^{p-1}$, as $j\to\infty$ (up to a constant),
by (\ref{finitesum}). Since $\mathfrak f$ is log-convex of degree $m,$ then $D^{p}f>0$ and so 
$\sum_{j=1}^{\infty} D^{p}f(j)(j+1)^{p-1}<\infty$ by Lemma \ref{represen}. On the other hand
$(l+j-v+1)^{p-1}(j+1)^{1-p}
\le \left((l/(j+1))+1\right)^{p-1}\le(l+1)^{p-1}$, if  $j,l>  v\geq 1$.

Therefore
\begin{displaymath}\begin{array}{l}
\displaystyle\sum_{j=v+1}^\infty D^p f(j)(j+1)^{p-1}
\sum_{l=v+1}^\infty k^p(l+j-v)(j+1)^{1-p}\vert W^{p-1} g(l)-W^{p-1}g(l+1)\vert \\
\displaystyle\leq M_p \sum_{j=v+1}^\infty D^p f(j) (j+1)^{p-1}
\sum_{l=v+1}^\infty\left({l+j-v+1\over j+1}\right)^{p-1}
(\vert W^{p-1}g(l)\vert+\vert W^{p-1}g(l+1)\vert)<\infty.
\end{array}\end{displaymath}

Also, for  $v\geq 1$ and $0\leq l\leq v$,  \begin{displaymath}\begin{array}{l}
\displaystyle\sum_{j=v-l}^v D^p f(j)k^p(l+j-v)=\displaystyle\sum_{j=v-l}^v D^{p-1} f(j)k^p(l+j-v)-\displaystyle\sum_{j=v-l+1}^{v+1} D^{p-1}f(j)k^p(l+j-v-1) \\
=\displaystyle\sum_{j=v-l}^v D^{p-1} f(j)k^{p-1}(l+j-v)-D^{p-1}f(v+1)k^p(l)
\end{array}\end{displaymath}
and for $l\geq  v+1$, \begin{displaymath}\begin{array}{l}
\displaystyle\sum_{j=v+1}^{\infty} D^p f(j)k^p(l+j-v)=\displaystyle\lim_{N}\biggl(\sum_{j=v+1}^N D^{p-1} f(j)k^p(l+j-v)-\displaystyle\sum_{j=v+2}^{N+1} D^{p-1}f(j)k^p(l+j-v-1)\biggr) \\
=\displaystyle\sum_{j=v+2}^{\infty} D^{p-1} f(j)k^{p-1}(l+j-v)+D^{p-1}f(v+1)k^p(l+1)-\lim_{N}D^{p-1} f(N+1)k^p(N+l-v)\\
=\displaystyle\sum_{j=v+2}^{\infty} D^{p-1} f(j)k^{p-1}(l+j-v)+D^{p-1}f(v+1)k^p(l+1).
\end{array}\end{displaymath}
In other words, the right hand side of \eqref{7} converges absolutely.

So, writing $D^{p}f(j)=D^{p-1}f(j)-D^{p-1}f(j+1)$ and $k^{p}(l+j-v)=k^{p-1}(l+j-v)-k^{p}(l+j-v-1)$ in \eqref{7}, one gets \begin{displaymath}\begin{array}{l}
\displaystyle\left(\sum_{l=0}^v\sum_{j=v-l}^v-\sum_{j=v+1}^\infty\sum_{l=v+1}^\infty\right)
k^p(l+j-v)D^p f(j)W^p g(l) \\
=\displaystyle\left(\sum_{l=0}^v\sum_{j=v-l}^v-\sum_{j=v+2}^\infty\sum_{l=v+1}^\infty\right)
k^{p-1}(l+j-v)D^{p-1} f(j)(W^{p-1} g(l)-W^{p-1}g(l+1))\\
\displaystyle-D^{p-1}f(v+1)\sum_{l=0}^v W^{p}g(l)k^p(l)-D^{p-1}f(v+1)\sum_{l=v+1}^{\infty} W^{p}g(l)k^p(l+1),
\end{array}\end{displaymath} where the series in the right hand member of the equality converge because $\mathfrak f$ is $p-1$-admissible.

Then, rearranging terms,

\begin{displaymath}\begin{array}{l}
\displaystyle\left(\sum_{l=0}^v\sum_{j=v-l}^v-\sum_{j=v+1}^\infty\sum_{l=v+1}^\infty\right)
k^p(l+j-v)D^p f(j)W^p g(l) \\
=\displaystyle\left(\sum_{l=0}^v\sum_{j=v-l}^v-\sum_{j=v+1}^\infty\sum_{l=v+1}^\infty\right)
k^{p-1}(l+j-v)D^{p-1} f(j)W^{p-1} g(l)\\
-\displaystyle\left(\sum_{l=0}^{v+1}\sum_{j=v+1-l}^{v+1}-\sum_{j=v+2}^\infty\sum_{l=v+2}^\infty\right)
k^{p-1}(l+j-v-1)D^{p-1} f(j)W^{p-1} g(l)\\
+\displaystyle D^{p-1}f(v+1)(\sum_{l=0}^{v+1} W^{p-1}g(l)k^{p-1}(l)+\sum_{l=v+1}^{\infty} W^{p-1}g(l)k^{p-1}(l+1))\\
\displaystyle-D^{p-1}f(v+1)(\sum_{l=0}^v W^{p}g(l)k^p(l)+\sum_{l=v+1}^{\infty} W^{p}g(l)k^p(l+1)).
\end{array}\end{displaymath}

By Proposition \ref{Porfin}, the series in the two above first brackets are $$W^{m-1}(f*g)(v)-W^{m-1}(f*g)(v+1)=W^{m}(f*g)(v).$$ Furthermore, note that $g(0)=\sum_{l=0}^{\infty}Wg(l)=\ldots=\sum_{l=0}^{\infty}W^{p}g(l)k^p(l)$ since $W^{p-1}g(l)k^{p}(l)\to 0$ as $l\to\infty.$ Then

\begin{displaymath}\begin{array}{l}
\displaystyle\sum_{l=0}^{v+1} W^{p-1}g(l)k^{p-1}(l)+\sum_{l=v+1}^{\infty} W^{p-1}g(l)k^{p-1}(l+1))-\sum_{l=0}^v W^{p}g(l)k^p(l)-\sum_{l=v+1}^{\infty} W^{p}g(l)k^p(l+1)\\
=\displaystyle\sum_{l=0}^{v+1} W^{p-1}g(l)k^{p-1}(l)+\sum_{l=v+1}^{\infty} W^{p-1}g(l)k^{p-1}(l+1))-g(0)\\
\displaystyle-\sum_{l=v+1}^{\infty}(W^{p-1}g(l)-W^{p-1}g(l+1))(k^p(l+1)-k^{p}(l))
=\displaystyle\sum_{l=0}^{\infty} W^{p-1}g(l)k^{p-1}(l)-g(0)=0,\\
\end{array}\end{displaymath}
and the result follows.
\end{proof}

{\it Proof of Theorem \ref{logconvex}.} By \cite[Prop. 4.4]{HT}, we have that $\mathfrak{f}$ is $0$-admissible. 
Let now assume that $\mathfrak{f}$ is $p-1$-admissible for some $p\in\N$ such that $p\leq m.$ 
We claim that $\mathfrak{f}$ is also $p$-admissible. First note that $\frak f$ has no zeroes in $\D$. Also, 
since $\mathfrak{f}$ is $p-1$-admissible, $D^{p}g(0)=D^{p-1}g(0)-D^{p-1}g(1)>0.$ 
Next, we prove that $W^{p}f(v)<0$ for every $v\in\N,$ by induction on $v.$ By Lemma \ref{Lemma7.6} 
we have, for $v=1$, 
$$
\sum_{l=0}^1W^{p}g(l)\sum_{j=1-l}^{\infty}k^{p}(j+l-1)D^{p}f(j)
=\sum_{l=0}^{\infty}W^{p}g(l)\sum_{j=2}^{\infty}k^{p}(j+l-1)D^{p}f(j).
$$ 

Then, since $W^{-p}(D^pf)=f$ one gets\begin{eqnarray*}
W^{p}g(1)f(0)&=&\sum_{l=0}^{\infty}(W^{p-1}g(l)-W^{p-1}g(l+1))\sum_{j=2}^{\infty}k^{p}(j+l-1)D^{p}f(j)-W^{p}g(0)f(1)\\
&=&\sum_{l=0}^{\infty}W^{p-1}g(l)\sum_{j=2}^{\infty}k^{p}(j+l-1)D^{p}f(j)\\
&&-\sum_{l=1}^{\infty}W^{p-1}g(l)\sum_{j=2}^{\infty}k^{p}(j+l-2)D^{p}f(j)-W^{p}g(0)f(1)\\
&=&W^{p-1}g(0)\sum_{j=2}^{\infty}k^{p}(j-1)D^{p}f(j)\\
&&+\sum_{l=1}^{\infty}W^{p-1}g(l)\sum_{j=2}^{\infty}k^{p-1}(j+l-1)D^{p}f(j)-W^{p}g(0)\sum_{j=1}^{\infty}k^{p}(j-1)D^{p}f(j)\\
&=&W^{p-1}g(1)f(1)-W^{p-1}g(0)D^{p}f(1)+\sum_{l=1}^{\infty}W^{p-1}g(l)\sum_{j=2}^{\infty}k^{p-1}(j+l-1)D^{p}f(j)<0,
\end{eqnarray*}
and therefore $W^{p}g(1)<0.$

Suppose $W^{p}g(l)<0$ for $1\leq l\leq v,$ with $v\in\N.$ We must show that $W^{p}g(v+1)<0$.

By Lemma \ref{Lemma7.6}, applied to $v$ and $v+1,$ we have 

\begin{equation*}\left\{
\begin{array}{l}
\displaystyle\sum_{l=0}^v W^p g(l)\sum_{j=v-l}^v k^p(j+l-v)D^p f(j)-\sum_{l=v+1}^\infty W^p g(l)\sum_{j=v+1}^\infty k^p(j+l-v) D^p f(j)=0\\ \\
\displaystyle\sum_{l=0}^{v+1} W^p g(l)\sum_{j=v-l}^v k^p(j+l-v)D^p f(j+1)-\sum_{l=v+2}^\infty W^p g(l)\sum_{j=v+1}^\infty k^p(j+l-v) D^p f(j+1)=0,
\end{array}\right.
\end{equation*}
where we have rearranged indexes in the second equality.
Multiplying the first identity by $D^{p}f(v+1)$ and the second by $D^{p}f(v)$ and subtracting one obtains 

\begin{equation}\label{7.4}
\begin{array}{l}
\displaystyle\sum_{l=1}^v W^p g(l)\sum_{j=v-l}^v k^p(j+l-v)(D^p f(j)D^pf(v+1)-D^{p}f(j+1)D^p f(v))\\ \\
\displaystyle-D^{p}f(v)W^{m}g(v+1)\sum_{l=0}^{v+1}D^p f(j)k^{p}(j)-D^{p}f(v+1)\sum_{l=v+1}^\infty W^p g(l)\sum_{j=v+1}^\infty k^p(j+l-v) D^p f(j)\\
\displaystyle+D^{p}f(v)\sum_{l=v+2}^\infty W^p g(l)\sum_{j=v+1}^\infty k^p(j+l-v) D^p f(j+1)=0.
\end{array}
\end{equation}

In \eqref{7.4}, the terms multiplying $W^p g(v+1)$ are
\begin{equation*}
\begin{array}{l}
\displaystyle -D^{p}f(v)\sum_{j=0}^{v+1}D^p f(j)k^{p}(j)-D^{p}f(v+1)\sum_{j=v+1}^\infty k^p(j+1) D^p f(j)\\ \\
=\displaystyle D^{p}f(v)\sum_{j=v+2}^{\infty}D^p f(j)k^{p}(j)-D^{p}f(v)f(0)-D^{p}f(v+1)\sum_{j=v+1}^\infty k^p(j+1) D^p f(j)\\ \\
=\displaystyle -D^{p}f(v)f(0)+\sum_{j=v+2}^{\infty}k^p(j)\biggl(D^pf(j)D^{p}f(v)-D^{p}f(j-1)D^{p}f(v+1)\biggr).
\end{array}
\end{equation*}

Hence, rewriting \eqref{7.4} suitably, one obtains

\begin{eqnarray*}
W^{p}g(v+1)f(0)&=&
\displaystyle\sum_{l=1}^v W^p g(l)
\sum_{j=v-l}^v k^p(j+l-v)D^p f(j)\biggl(\frac{D^pf(v+1)}{D^{p}f(v)}-\frac{D^{p}f(j+1)}{D^p f(j)}\biggr)\\
&&
-\displaystyle\sum_{l=v+1}^{\infty} W^p g(l)
\sum_{j=v+1}^{\infty} k^p(j+l-v)D^p f(j)\biggl(\frac{D^pf(v+1)}{D^{p}f(v)}-\frac{D^{p}f(j+1)}{D^p f(j)}\biggr)\\
&=&
\displaystyle\sum_{l=1}^v W^p g(l)\sum_{j=v-l}^v k^p(j+l-v)D^p f(j)\biggl(\frac{D^pf(v+1)}{D^{p}f(v)}-\frac{D^{p}f(j+1)}{D^p f(j)}\biggr)\\
&&
-\displaystyle\sum_{l=v+2}^{\infty} W^{p-1} g(l)
\sum_{j=v+1}^{\infty} k^{p-1}(j+l-v)D^p f(j)\biggl(\frac{D^pf(v+1)}{D^{p}f(v)}-\frac{D^{p}f(j+1)}{D^p f(j)}\biggr)\\
&&
-W^{p-1} g(v+1)\sum_{j=v+1}^{\infty} k^{p}(j+1)D^p f(j)\biggl(\frac{D^pf(v+1)}{D^{p}f(v)}-\frac{D^{p}f(j+1)}{D^p f(j)}\biggr),
\end{eqnarray*}
whence the induction hypothesis and the log-convexity of degree $m$ of $\mathfrak{f}$ imply $W^{p}g(v+1)<0,$ as we wanted to show.

Note that we have proved that $W^{p}g(0)>0$ and $W^{p}g(j)<0$ for $j\in\N$ and $p=0,1,\ldots,m.$ 
Hence, the bounded and free-zero function $\mathfrak f$  is $p$-admissible for every $p=0,1,\ldots,m.$ 
In particular 
$\mathfrak{g}\in A^{m}(\D)\subset A^{r}(\D),$ for $0\leq r\leq m,$ and so 
$W^{r}g=D^{r}g.$ Take $p\in\N$ with $p\leq m.$ If $\beta \in(p-1,p)$ we have 
\begin{eqnarray*}
W^{\beta}g(0)&=&D^{\beta}g(0)=\sum_{l=0}^{\infty}k^{p-1-\beta}(l)D^{p-1}g(l)\\
&=&k^{p-1-\beta}(0)D^{p-1}g(0)+\sum_{l=1}^{\infty}k^{p-1-\beta}(l)D^{p-1}g(l)>0,
\end{eqnarray*}
and for $j\in\N,$ 
$$
W^{\beta}g(j)=\sum_{l=0}^{\infty}k^{p-1-\beta}(l)D^{p-1}g(l+j)=\sum_{l=1}^{\infty}k^{p-1-\beta}(l)(D^{p-1}g(l+j)-D^{p-1}g(j))<0,
$$
since $D^{p-1}g$ is increasing on $\N$ because $D^{p}g<0$ on $\N.$ 
All the above together with Remark \ref{remark7.1} imply that $\mathfrak{f}$ is $\beta$-admissible for each 
$0\leq \beta\leq m.$ $\hfill {\Box}$

\end{document}